\documentclass[12pt]{amsart}
\usepackage{enumerate}
\usepackage{amsmath, amssymb, epsfig, graphicx, url}
\usepackage{color,todonotes}
\usepackage{fullpage} 
\newtheorem{thm}{Theorem}[section]
\newtheorem{lem}[thm]{Lemma}
\newtheorem{cor}[thm]{Corollary}
\newtheorem{prop}[thm]{Proposition}
\newtheorem{exam}[thm]{Example}
\newtheorem{defn}[thm]{Definition}
\newtheorem{problem}{Problem}

\renewcommand{\leq}{\leqslant}
\renewcommand{\geq}{\geqslant}

\newcommand{\aut}{\mbox{\rm Aut\,}}

\newcommand{\rank}{\operatorname{rank}}

\newcommand{\genset}[1]{\ensuremath{\langle\: #1 \:\rangle}}
\newcommand{\trans}{\mathcal{T}_{n}}
\newcommand{\sym}{\mathcal{S}_{n}}
\newcommand{\alt}{\mathcal{A}_{n}}

\newcommand{\agl}{\mbox{\rm AGL}}
\newcommand{\psl}{\mbox{\rm PSL}}
\newcommand{\pgl}{\mbox{\rm PGL}}
\newcommand{\asl}{\mbox{\rm ASL}}
\newcommand{\pgaml}{\mbox{\rm P}\Gamma {\rm L}}
\newcommand{\agaml}{\mbox{\rm A}\Gamma {\rm L}}
\newcommand{\psigl}{\mbox{\rm P}\Sigma {\rm L}}
\newcommand{\pgamu}{\mbox{\rm P}\Gamma {\rm U}}
\newcommand{\Sp}{\mbox{\rm Sp}}
\newcommand{\dihed}[1]{\ensuremath{D_{#1}}}

\newcommand{\M}{\mbox{\rm M}}
\newcommand{\GF}{\mbox{\rm GF}}

\newcommand{\G}{{\mbox{\rm G}}}
\newcommand{\Co}{{\mbox{\rm Co}}}
\newcommand{\Sz}{\mbox{\rm Sz}}
\newcommand{\HS}{\mbox{\rm HS}}

\newcommand{\ut}[1]{has the ${#1}$-ut property}
\newcommand{\notut}[1]{does not have the ${#1}$-ut property}

\renewcommand{\to}{\longrightarrow}

\begin{document}



\title[Groups and regular semigroups]{Two Generalizations of Homogeneity in Groups with Applications to Regular Semigroups}
\author{Jo\~{a}o Ara\'{u}jo}
\author{Peter J. Cameron}

\address[Ara\'{u}jo]
{Universidade Aberta
and
Centro de \'{A}lgebra \\
Universidade de Lisboa \\
Av. Gama Pinto, 2, 1649-003 Lisboa \\ Portugal}
\email{\url{jaraujo@ptmat.fc.ul.pt}}

\address[Cameron]{Department of Mathematics \\
School of Mathematical Sciences at Queen Mary\\ University of London}
\email{\url{P.J.Cameron@qmul.ac.uk }}

\maketitle


\begin{abstract}
Let $X$ be a finite set such that $|X|=n$ and let $i\leq j \leq n$. A group $G\leq \sym$ is said to be $(i,j)$-homogeneous if for every $I,J\subseteq X$, such that $|I|=i$ and $|J|=j$, there exists $g\in G$ such that $Ig\subseteq J$. (Clearly $(i,i)$-homogeneity is $i$-homogeneity in the usual sense.)

A group $G\leq \sym$ is said to have the $k$-universal transversal property if given any set $I\subseteq X$ (with $|I|=k$) and any partition $P$ of $X$ into $k$ blocks, there exists $g\in G$ such that $Ig$ is a section for $P$. (That is, the orbit of each $k$-subset of $X$ contains a section for each $k$-partition of $X$.)

In this paper we classify the groups with the $k$-universal transversal property (with the exception of two classes of $2$-homogeneous groups) and the $(k-1,k)$-homogeneous groups (for $2<k\leq \lfloor \frac{n+1}{2}\rfloor$). As a corollary of the classification we prove that a $(k-1,k)$-homogeneous group is also $(k-2,k-1)$-homogeneous, with two exceptions; and similarly, but with no exceptions, groups having the $k$-universal transversal property have the $(k-1)$-universal transversal property.  

A corollary of all the previous results is a  classification of  the groups that together with any rank $k$ transformation on $X$ generate a regular semigroup (for $1\leq k\leq \lfloor \frac{n+1}{2}\rfloor$).

The paper ends with a number of challenges for experts in number theory, group and/or semigroup theory, linear algebra and matrix theory.  
\end{abstract}

\medskip

\noindent{\em Date:} 10 August 2011
{\em Corresponding author: Jo\~{a}o Ara\'{u}jo, jaraujo@ptmat.fc.ul.pt}





\newpage

\section{Introduction and Preliminaries}

One of the fundamental trends in
 semigroup theory has been the study of how idempotents shape the structure of the semigroup. Howie's book \cite{Ho95} can be seen as an excellent survey of the results obtained from the 40s to the 90s  on this general problem. Evidently, there is the analogous question for the group of units, namely, to what extent the group of units shapes the structure of the semigroup. (And a similar question can be asked about the normalizer of the semigroup; more on this below.) However, unlike the idempotents case, the group of units approach quickly leads to problems that could not be tackled with the tools available 30 years ago, let alone 70 years ago. Fortunately now the situation is totally different, since the enormous  progress made  in the last decades in the theory of permutation groups  provides   the necessary tools to develop semigroup theory from this different point of view.

A particular instance of the general problem of investigating how the group of units shapes the whole semigroup might be described as follows:
{\em classify the pairs $(a,G)$, where $a$ is a map on $X$ and $G$ is a group of permutations of $X$,
such that the semigroup  $\genset{a,G}$, generated by $a$ and $G$,  has a given property}.  (Observe that whenever $S$ is a semigroup with group of units $G$ we have  $S=\cup_{a\in S}\genset{a,G}$.)

A very important class of groups that falls under this general scheme is that of synchronizing groups, groups of permutations on a set that together with any non-invertible map on the same set generate a constant (see \cite{steinberg}, \cite{cameron}, \cite{neumann}). These groups are very interesting from a group theoretic  point of view and are linked  to the \v{C}ern\'y conjecture,  a longstanding open problem in automata theory.

Another instance of the general problem described above is the following: {\em classify the permutation groups on a set that together with any map on that set generate a regular semigroup}. (An element $a$ in a semigroup $S$ is said to be regular if there exists $b\in S$ such that $a=aba$. The semigroup is said to be regular if all its elements are regular.) This question has been answered in \cite{ArMiSc} as follows. (From now on $\sym$ will denote the symmetric group  on the set $[n]:=\{1,\ldots,n\}$; by $\trans$ we will denote the full transformation monoid on  $[n]$. We use the notation $D(2*n)$ for the dihedral group of order $2n$, called either $D_{n}$ or $D_{2n}$ in the literature; other notation for finite groups is standard.) 

\begin{thm}\label{th2}
If $n\geq 1$ and $G$ is a subgroup of $\sym$,  then the following are equivalent:
\begin{enumerate}
    \item[\rm (i)] The semigroup $\genset{G,a}$ is regular for all $a\in \trans$.
 \item[\rm (ii)] One of the following is valid for $G$ and $n$:
    \begin{enumerate}
    \item[\rm (a)] $n=5$ and $G\cong C_5,\  D(2*5),$ or $\agl(1,5)$;
    \item[\rm (b)] $n=6$ and $G\cong \psl(2,5)$ or $\pgl(2,5)$;
    \item[\rm (c)]  $n=7$ and $G\cong\agl(1,7)$;
    \item[\rm (d)] $n=8$ and $G\cong\pgl(2,7)$;
    \item[\rm (e)] $n=9$ and $G\cong\psl(2,8)$ or $\pgaml(2,8)$;
    \item[\rm (f)] $G=\alt$ or $\sym$.
\end{enumerate}
\end{enumerate}
\end{thm}

The critical observation that led to  the proof of Theorem \ref{th2} is that for $n\geq 12$, if $G\ (\leq \sym)$ satisfies the property that  any rank $\lfloor\frac{n}{2}\rfloor$ map $a\in \trans$ is regular in the semigroup $\genset{G,a}$, then $G$ contains the alternating group. Therefore Theorem \ref{th2} can be seen as an (almost) immediate corollary of the following result.

\begin{thm}\label{th3}
If $n\geq 12$ and $G$ is a subgroup of $\sym$,  then the following are equivalent:
\begin{enumerate}
    \item[\rm (i)] The element $a$ is regular in $\genset{G,a}$, for all $a\in \trans$ such that  $\rank(a)=\lfloor\frac{n}{2}\rfloor$.
\item[\rm (ii)] $G=\alt$ or $\sym$.
\end{enumerate}
\end{thm}
Below a sharper version of this result is going to be stated.

The aim of this paper is to carry out a deeper group theoretic analysis in order to prove a lemma on the groups $G$  that satisfy the following property: for every rank $k$ map $a$  (with fixed $k$ such that $1\leq k\leq \lfloor\frac{n+1}{2}\rfloor$), $a$ is regular in $\genset{G,a}$; and then extract many important consequences from this result.

In order to understand these groups we need some observations. Suppose $G\leq \sym$ and all rank $k$ maps $a\in \trans$ are regular in  $\genset{G,a}$. Then there exists $b\in  \genset{G,a}$ such that $a=aba$ and hence $ab$ is an  idempotent in $\genset{G,a}$ having the same rank as $a$, where $b=a$ or $b\in G$ or $b=g_{1}ag_{2}a\ldots ag_{m}$, for $m\geq 2$. Suppose $b\in G$. Then there exists $g\in G$ (namely $g=b$) such that $ag$ is idempotent and hence $[n]ab$ is a section for the partition of $[n]$ induced by the kernel of $a$. Suppose $b=g_{1}ag_{2}a\ldots ag_{m}$ with $m\geq 2$. Since $ab$ is idempotent and  $ab=ag_{1}ag_{2}a\ldots ag_{m}$ it follows that  there exists $g\in G$ (namely $g=g_{1}$) such that $a$ and $aga$ have the same rank. In both cases there exists $g\in G$ such that  $[n]ag$ is a section for the partition of $[n]$ induced by the kernel of $a$. As this property must hold for all rank $k$ maps $a\in \trans$, it follows that $G$ must satisfy the following $k$-\textit{universal transversal property}: for every $k$-set $I\subseteq [n]$  and every partition  $P$ of $[n]$ into $k$ blocks, there exists $g\in G$ such that $Ig$ is a section for $P$. The next six results (almost) provide  the  classification of the groups possessing the $k$-universal transversal property. 

The first theorem handles the permutation groups of small degree. 

\begin{thm}\label{th4a}
For $n<11$ and $2\leq k\leq \lfloor\frac{n+1}{2}\rfloor$, a group  $G\leq \sym$ with the $k$-universal transversal property is
$k$-homogeneous, with the following exceptions: 
\begin{enumerate}
\item $n=5$,  $G\cong C_{5}$ or $D(2*5)$ and $k=2$;
\item $n=6$, $G\cong \psl(2,5)$ and $k=3$;
\item  $n=7$,  $G\cong C_{7}$ or $G\cong D(2*7)$, and $k=2$; or 
$G\cong\agl(1,7)$ and $k=3$;
\item $n=8$, $G\cong \pgl(2,7)$ and $k=4$;
\item  $n=9$,  $G\cong 3^{{2}}:4$ or $G\cong 3^{{2}}:D(2*4)$ and $k=2$;
\item $n=10$, $G\cong \mathcal{A}_{5} $ or $G\cong \mathcal{S}_{5}$ and $k=2$; or 
$G\cong\psl(2,9)$ or $G\cong \mathcal{S}_{6}$   and $k=3$.
\end{enumerate}
\end{thm}

The next results deal with the groups of degree larger than 10. We start by the case of groups possessing the $k$-universal transversal property, for large values of $k$. 

\begin{thm}\label{th4b}
Let $n\geq 11$, $G\leq \sym$.
If $6\leq k\leq \lfloor \frac{n+1}{2}\rfloor$, then the following are equivalent: 
		\begin{enumerate}
			\item $G$ has the $k$-universal transversal property; 
			\item $\alt\leq G$. 
		\end{enumerate}
\end{thm}

The four next results deal with groups possessing the $k$-universal transversal property, when $k\in \{2,\ldots,5\}$. 

\begin{thm}\label{th4c}
Let $n\geq 11$, $G\leq \sym$ and let $2\leq k\leq \lfloor \frac{n+1}{2}\rfloor$.
The following are equivalent: 
		\begin{enumerate}
			\item $G$ has the $5$-universal transversal property; 
			\item $G$ is 5-homogeneous, or  $n=33$ and $G=\pgaml(2,32)$.
		\end{enumerate}
\end{thm}

Unlike the previous cases, the classification of groups possessing the $4$-universal transversal property was not possible. So far we have the following results, and we believe the remaining cases require very delicate considerations. 

\begin{thm}\label{th4d}
Let $n\geq 11$, $G\leq \sym$ and let $2\leq k\leq \lfloor \frac{n+1}{2}\rfloor$. If $G$ is $4$-homogeneous,  or $n=12$ and $G=\M_{11}$, then $G$ has the $4$-universal transversal property. 

If there are more groups possessing the $4$-universal transversal property, then they must be groups $G$ such that $\psl(2,q)\leq G\leq \pgaml(2,q)$, with either $q$ prime  or $q=2^p$ for $p$ prime. Note, however, that  the groups $\psl(2,q)$ for $q\equiv1$ (mod~$4$) cannot possess the $4$-universal transversal property since  they fail to satisfy the necessary condition of $3$-homogeneity.  
\end{thm}

Similarly to the previous case, a full classification of the groups possessing the $3$-universal transversal property was not possible.
\begin{thm}\label{th4e}
Let $n\geq 11$, $G\leq \sym$ and let $2\leq k\leq \lfloor \frac{n+1}{2}\rfloor$. $G$ has the $3$-universal transversal property if $G$ is $3$-homogenous,  or one of the following groups 
					\begin{enumerate}
						\item $\psl(2,q)\leq G\leq \psigl(2,q)$, where $q\equiv 1 (\mbox{mod }4)$; 
						\item $\Sp(2d,2)$ with $d\geq 3$, in either of its $2$-transitive representations;
						\item $2^{{2d}}:\Sp(2d,2)$;
						\item \HS;
						\item $\Co_{3}$;
						\item $2^{{6}}:\G_{2}(2)$ and its subgroup of index $2$;
						\item $\agl(1,p)$ where, for all $c\in {\GF}(p)\setminus\{0,1\}$, $|\langle -1,c,c-1\rangle|=p-1$. 
						\end{enumerate}
						
 If there are more groups possessing the $3$-universal transversal property, then they must be Suzuki groups $\Sz(q)$, possibly with field automorphisms adjoined, and/or subgroups of index $2$ in $\agl(1,p)$ for $p\equiv 11$ $(mod\ 12)$.
\end{thm}
 
 Finally, possessing the $2$-universal transversal property is just another way of saying primitive. 
 
 \begin{thm}\label{th4f}
 A permutation group has the $2$-universal transversal property if and only the group is primitive. 
 \end{thm}

These theorems immediately imply an analogue of the Livingstone--Wagner \cite{lw} result on $k$-homogenous groups (for $2\leq k\leq \lfloor\frac{n+1}{2}\rfloor$). 

\begin{cor}\label{cor1}
Let $n\geq 5$, let $2\leq k\leq \lfloor \frac{n+1}{2}\rfloor$, and let $G\leq \sym$ be a group having the $k$-universal transversal property. Then $G$ has the $(k-1)$-universal transversal property.  
\end{cor}

Now consider partitions of type $(\overbrace{1,\ldots,1}^{k-1},n-k+1)$ (that is, $k-1$ classes of size 1 and one class of size $n-k+1$). Any group $G$ satisfying the $k$-universal transversal property must also satisfy the following property: for every $k$-partition $P$ of $[n]$ of type
$({1,\ldots,1},n-k+1)$ and for every $k$-set $I$ there exists 
$g\in G$ such that $Ig$ is a section for  $P$. In particular this implies that the union of all the singleton blocks is contained in $Ig$; as the union of the singleton blocks can be any $(k-1)$-subset of $[n]$, it follows that any group $G$ possessing the $k$-universal transversal property must be $(k-1,k)$-homogeneous, that is, for every $(k-1)$-set $I$ and for every $k$-set $J$ there exists $g\in G$ such that $Ig\subseteq J$.   In this new setting we can state the sharper version of Theorem \ref{th3} announced above.

\begin{thm}\label{th5}
If $n\geq 12$ and $G$ is a subgroup of $\sym$,  then the following are equivalent:
\begin{enumerate}
    \item[\rm (i)] $G$ is $(\lfloor \frac{n}{2}\rfloor-1,\lfloor \frac{n}{2}\rfloor)$-homogeneous.
\item[\rm (ii)] The map $a$ is regular in  $\genset{G,a}$ for all $a\in \trans$ such that $\rank (a)=\lfloor\frac{n}{2}\rfloor$ and $a$ has kernel type $(1,\ldots,1,\lceil\frac{n}{2}\rceil+1)$.
    \item[\rm (iii)] The map $a$ is regular in  $\genset{G,a}$ for all $a\in \trans$ such that $\rank(a)=\lfloor\frac{n}{2}\rfloor$.
\item[\rm (iv)] $G=\alt$ or $\sym$.
\end{enumerate}
\end{thm}

Our second main theorem on groups provides the following classification of the $(k-1,k)$-homogeneous groups.
\begin{thm}\label{th6}
If $n\geq 1$ and $2\leq k\leq \lfloor \frac{n+1}{2} \rfloor$, then the following are equivalent:
\begin{enumerate}
    \item[\rm (i)] $G$ is a $(k-1,k)$-homogeneous subgroup of $\sym$;
    \item[\rm (ii)] $G$ is $(k-1)$-homogeneous or $G$ is one of the following groups
    \begin{enumerate}
    \item[\rm (a)] $n=5$ and $G\cong C_5,\ D(2*5),$  $k=3$; 
    \item[\rm (b)]  $n=7$ and $G\cong\agl(1,7)$, with $k=4$;
    \item[\rm (c)]  $n=9$ and $G\cong\asl(2,3)$ or $\agl(2,3)$, with $k=5$. 
\end{enumerate}
\end{enumerate}
\end{thm}

Once again an analogue of the Livingstone--Wagner \cite{lw} result is immediate.

\begin{cor}\label{cor2}
Let $n\geq 1$, let $3\leq k\leq \lfloor \frac{n+1}{2}\rfloor$, and let $G\leq \sym$ be a $(k-1,k)$-homogeneous group. Then $G$ is a $(k-2,k-1)$-homogeneous group, except when  $n=9$ and $G\cong\asl(2,3)$ or $\agl(2,3)$, with $k=5$.  
\end{cor}

By \cite[Theorem 2.3]{lmm} we know that every rank $k$ map $a\in \trans$ is regular in $\genset{G,a}$ if and only if $G$ has the $k$-universal transversal property, that is, in the orbit of every $k$-set there exists a transversal (or section) for every $k$-partition. Therefore we can state our main results about semigroups.

A quasi-permutation is a transformation in which all kernel classes but one are singletons.  
\begin{thm}\label{th6c}
Let $G\leq \sym$ and let $1< k\leq \lfloor \frac{n+1}{2}\rfloor$. Then the following are equivalent
\begin{enumerate}
\item for every quasi-permutation $a$, such that $\mbox{rank}(a)=k$, the semigroup $\langle G,a\rangle$ is regular;
\item $G$ is $(k-1)$-homogeneous or $G$ is one of the following groups
   \begin{enumerate}
    \item[\rm (a)] $n=5$ and $G\cong C_5,\ D(2*5)$, with $k=3$;
    \item[\rm (b)]  $n=7$ and $G\cong\agl(1,7)$, with $k=4$;
\end{enumerate}
\end{enumerate}
\end{thm}

We are now ready to state  the result that dramatically generalizes Theorem \ref{th2} (and also \cite[Theorem 2.3]{lmm} taking advantage of the fact that if a group has the $k$-universal transversal property, then it also has the $(k-1)$-universal transversal  property, by Corollary \ref{cor1}).

\begin{thm}\label{Kut1}
Let $n\geq 5$, $G\leq \sym$ and let $1< k\leq \lfloor \frac{n+1}{2}\rfloor$. Then the following are equivalent:
\begin{enumerate}
	\item all rank $k$ transformations $a\in \trans$ are regular in $\langle a, G\rangle$;
	\item for all rank $k$ transformations $a\in \trans$, the semigroup $\langle a, G\rangle$ is regular;
	\item $G$ has the $k$-universal transversal property (and hence is one of the groups listed in the classification).
\end{enumerate}
\end{thm}

One last word about the normalizer. It is well known that not every semigroup has a group of units, and hence the approach proposed in this paper might seem limited. Therefore two observations should be made here. The first is that it is commonly believed  that the majority of finite semigroups have only one idempotent (which is a zero), but that did not prevent experts in semigroup theory to investigate how idempotents shape the structure of a semigroup; and the second observation is that by \cite[Theorem 2.3, $(iii)\Leftrightarrow(iv)\Leftrightarrow (v)$]{lmm}, Theorem \ref{th6c} and Theorem \ref{Kut1}  admit versions in terms of conjugates.  As a sample result we have the following immediate (from \cite[Theorem 2.3]{lmm}) version of the previous theorem.  

\begin{thm}\label{KutNorm}
Let $n\geq 5$, $G\leq \sym$ and let $1< k\leq \lfloor \frac{n+1}{2}\rfloor$. Then the following are equivalent:
\begin{enumerate}
	\item all rank $k$ transformations $a\in \trans$ are regular in $$\langle g^{-1}ag\mid g\in G\rangle;$$
	\item for all rank $k$ transformations $a\in \trans$ the semigroup  $\langle g^{-1}ag\mid g\in G\rangle$ is regular;
	\item $G$ has the $k$-universal transversal property.
\end{enumerate}
\end{thm}
  
  This observation is important because the transformation semigroup $S\leq \trans$ might contain no group of units, but every transformation semigroup $S$ has a normalizer and hence the results of this paper can be used to extract information about the structure of $S$ from its normalizer. For example, if $S=\langle t_{1},\ldots ,t_{m}\rangle$ is a semigroup generated by $m$ rank 3 maps (for example in $\mathcal{T}_{176}$) and it turns out that the normalizer $N(S):=\{g\in\sym \mid  g^{-1}Sg = S\}$ contains the Higman--Sims group, then we know that the semigroup $S$ is regular.   
  
 An even more striking consequence of Corollary \ref{cor1} and of the fact that  possessing the $k$-universal transversal property is closed upwards (that is, if $G\leq H\leq \sym$ and $G$ has the $k$-universal transversal property, then $H$ also has it), is the following result. 
  
  \begin{thm}
  Let $n\geq 5$, $G\leq \sym$ and let $1< k< \lfloor \frac{n+1}{2}\rfloor$. Then the following are equivalent:
\begin{enumerate}
	\item[(i)] $G$ has the $k$-universal transversal property;
	\item[(ii)] $G$ has the $l$-universal transversal property for all $l$ such that  $1\leq l\leq k$; 
	\item[(iii)] $H$ has the $k$-universal transversal property for all $H$ such that $G\leq H\leq \sym$;
	\item[(iv)] $H$ has the $l$-universal transversal property for all $H$ such that $G\leq H\leq \sym$ and for all $l$ such that  $1\leq l\leq k$;
	\end{enumerate}
	
	As a consequence, the following are equivalent.
\begin{enumerate}
	\item all rank $k$ transformations $a\in \trans$ are regular in $\langle a, G\rangle$;
	\item all rank $k$ transformations $a\in \trans$ are regular in $\langle g^{-1}ag\mid g\in G\rangle$;
\item for all rank $k$ transformations $a\in \trans$ and for all groups $H$ such that $G\leq H\leq \sym$ we have that  
$\langle h^{-1}ah\mid h\in H\rangle$ is a regular semigroup.
\item for all rank $k$ transformations $a\in \trans$ and for all groups $H$ such that $G\leq H\leq \sym$ we have that  
$\langle a, H\rangle$ is a regular semigroup.
\end{enumerate}
\end{thm}

  Finally we summarize  what this paper brings to groups and to semigroups:
  
  \begin{enumerate}
  \item We have generalized the notion of $(k-1)$-homogeneity in permutation groups; we have extended it first to the obvious notion of $(k-1,k)$-homogeneous groups and then extended this to the notion of groups having the $k$-universal transversal property (for $2\leq k\leq \lfloor \frac{n}{2}\rfloor$). 
  \item The $(k-1,k)$-homogenous groups were fully classified, and the groups having the $k$-universal transversal property have been  classified, with the exception of a class of groups (for $k=3$) and another class (for $k=4$). These two classes left undecided are surely very interesting problems for group theorists and combinatorialists.   
  \item As a corollary of the classification  it follows that   $(k-1,k)$-homogenous groups are $(k-2,k-1)$-homogenous with two exceptions; and groups having the $k$-universal transversal property have the $(k-1)$-universal transversal property. And this fact  is extremely important for the impact of these results on the theory of semigroups. 
  \item Regarding semigroups, we take deep results out of the classification of finite simple groups and show that it is possible to follow the promising path of investigating how the group of units (or other groups associated to the semigroup such as the normalizer) shape the structure of the semigroup. This mimics what has been done in semigroup theory for the last 70 years with the set of idempotents. 
 \item The paper ends with a number of challenges for experts in number theory, group and/or semigroup theory, linear algebra and matrix theory.  
  \end{enumerate}
  
\section{The classification of $(k-1,k)$-homogeneous groups}

A permutation group $G$ of degree~$n$ is \emph{$k$-homogeneous} if it acts
transitively on the set of $k$-subsets of its domain. Since $k$-homogeneity
is clearly equivalent to $(n-k)$-homogeneity, it is usually assumed that
$k\le(n-1)/2$. With this assumption, Livingstone and Wagner, in an elegant
paper~\cite{lw}, proved that a $k$-homogeneous group is $(k-1)$-homogeneous,
and is $k$-transitive if $k\ge5$. 
Kantor~\cite{kantor:4homog,kantor:2homog} determined all 
$k$-homogeneous groups which are not $k$-transitive for $2\le k\le 4$. 
The $2$-transitive groups have been determined as a consequence of the
Classification of Finite Simple Groups; lists of them can be found in 
\cite{cam} and \cite{dixon}.

For $k\le l$, the permutation group $G$ is $(k,l)$-homogeneous if, given
subsets $K,L$ of the domain with $|K|=k$ and $|L|=l$, there is an element
of $G$ which maps $K$ to a subset of $L$. Note that
\begin{itemize}
\item $(k,k)$-homogeneity is equivalent to $k$-homogeneity, and for fixed $k$
the concept of $(k,l)$-homogeneity   becomes formally weaker as $l$ increases;
\item $(k,l)$-homogeneity is equivalent to the ``dual'' concept of $(l,k)$-homogeneity, requiring that for given $K$ and $L$ as before, there is an
element of $G$ mapping $L$ to a superset of $K$;
\item $(k,l)$-homogeneity is equivalent to $(n-l,n-k)$-homogeneity.
\end{itemize}

In this section we are concerned with $(k,k+1)$-homogeneity. Because of the
third property above, we lose no generality in assuming, as Livingstone and
Wagner did, that $k\le(n-1)/2$; indeed this condition will be used in our
proofs several times. We had hoped to find arguments as elegant as those of
Livingstone and Wagner, but have not succeeded. We prove the following
theorem:

\begin{thm}\label{2.1}
Let $G$ be a $(k,k+1)$-homogeneous permutation group of degree $n\ge2k+1$.
Then either $G$ is $k$-homogeneous, or $G$ is one of the following groups:
\begin{enumerate}
\item $n=5$, $k=2$, $G\cong C_5$ or $D(2*5)$;
\item $n=7$, $k=3$, $G=\agl(1,7)$;
\item $n=9$, $k=4$, $G=\asl(2,3)$ or $\agl(2,3)$.
\end{enumerate}
\label{kkp1}
\end{thm}

\subsection{General observations}

Let $G$ be $(k,k+1)$-homogeneous of degree $n\ge2k+1$. We begin with a few
general observations.

\begin{prop}
The number of $G$-orbits on $k$-sets is at most $k$.
\end{prop}

\begin{proof}
Since a fixed $(k+1)$-set contains a representative of every orbit on 
$k$-sets, there are at most
$k+1$ orbits on $k$-sets. This bound can be reduced to $k$. For suppose there
are $k+1$ orbits; then each $(k+1)$-set contains exactly one $k$-set from each
orbit. Let $V_k$ be the $\mathbb{Q}$-vector space of functions from $k$-sets
to $\mathbb{Q}$, and let $T:V_k\to V_{k+1}$ be defined by
\[(Tf)(L)=\sum_{x\in L}f(L\setminus\{x\})\]
for $f\in V_k$ and $|L|=k+1$. Since $n\ge 2k+1$, it is known that $T$ is
injective (Kantor~\cite{kantor:inc}). However, if $f$ is the characteristic
function of any $G$-orbit, then $Tf$ is the all-$1$ function.
\end{proof}

This gives a lower bound for $|G|$, namely $|G|\ge{n\choose k}/k$. We refer
to this as the \emph{order bound}.
The right-hand side of this bound is a monotonic function of $k$
for $k<n/2$; so, whenever we rule out a group $G$ on the basis of this
inequality for a certain value of $k$, then it cannot occur for any larger
value of $k$ either.

The \emph{Ramsey number} $R(k,l,r)$, for positive integers $k,l,r$ with
$k\le l$ and $r\ge1$, is the smallest number $n$ such that, if the $k$-element
subsets of an $n$-set are coloured with $r$ colours, there exists a
$l$-element set all of whose $k$-element subsets have the same colour.

\begin{prop}
If $G$ is $(k,k+1)$-homogeneous but not $k$-homogeneous of degree $n$, then
$n<R(k,k+1,2)$.
\end{prop}

\begin{proof}
Colour the $k$-sets in one $G$-orbit
red and the others blue. Each $(k+1)$-set contains $k$-sets of each colour.
\end{proof}

It happens that the Ramsey numbers $R(2,3,2)=6$ and $R(3,4,2)=13$ are two
of the very few which are known exactly. The first is the well-known ``party
problem''; the second was computed by McKay and Radzizowski~\cite{mr} in 1991
(see \cite{rad} for a survey). The number $R(4,5,2)$ is not known, and the
known upper bounds are too large for our purpose.

In the case $k=2$, we have $n\le5$, and it is easy to see that the
cyclic and dihedral groups are examples and hence we have (1) of Theorem \ref{2.1}.  So we may assume that $k\ge3$.

\medskip

Our general results allowed us to decide which groups are $(k,k+1)$-homogeneous, except for a number of groups of small degrees. To decide those cases we used \textsf{GAP}~\cite{GAP} and include here a word about those computations. For $n\le20$,
the simplest method is to compute the orbits of a given group on $k$-sets
and $(k+1)$-sets, and for each orbit representative on $(k+1)$-sets, test
whether it contains representatives of all the $k$-set orbits.

For larger $n$, the memory requirements of this method are too heavy, so
we proceed a little differently. First, as we will prove below, any candidate
group must be $2$-transitive; so we reject groups which either fail to be
$2$-transitive or are $k$-homogeneous. We also reject groups which fail the
order bound $|G|\ge{n\choose k}/k$. Then, if $G$ is $t$-transitive, we loop
over all pairs $(K',L')$, where $K'$ and $L'$ are subsets of $\{t+1,\ldots,n\}$
of cardinality $k-t$ and $k+1-t$ respectively, and check whether there is
an element of $G$ mapping $\{1,\ldots,t\}\cup K'$ to a subset of 
$\{1,\ldots,t\}\cup L'$. If this fails for any pair $(K',L')$, we can terminate
the computation and report that $G$ is not $(k,k+1)$-homogeneous.

\subsection{Transitivity}

From now on, $G$ will be a $(k,k+1)$-homogeneous but not $k$-homogeneous 
group of degree $n\ge2k+1$, which is not one of the exceptions listed in 
the statement of Theorem~\ref{kkp1}.

\begin{prop}
$G$ is transitive.
\end{prop}

\begin{proof}
Let $O$ be an orbit of $G$. There exists a $k$-set containing at least
one point of $O$. Hence every $(k+1)$-set contains a $k$-set containing at
least one point of $O$, and thus intersects $O$. So $|O|\ge n-k > n/2$. Since
$O$ was arbitrary, there is only one orbit.
\end{proof}

\subsection{Primitivity}

\begin{prop}
$G$ is primitive.
\end{prop}

\begin{proof}
Suppose that $G$ is imprimitive, with $r$ blocks of size $s$. 

If $k\le s$, then there is a $k$-set contained in a block. But, since 
$k+1\ge4$, there is a $(k+1)$-set and all instances of an $k$-set or an $(k+1)$    containing at least two points of each of
two blocks; such a set cannot contain a $k$-set of the type just described.

So $k>s$, and $r=n/s>n/k>2$, so $r\ge3$. There is a $(k+1)$-set which 
contains either $\lfloor(k+1)/r\rfloor$ or $\lceil(k+1)/r\rceil$ points from
each block. On the other hand, there is a $k$-set containing all the points of
a block. So $(k+1)/r\ge\lceil(k+1)/r\rceil-1\ge s-1$, whence
$k+1\ge (r-1)s\ge2n/3$, a contradiction.
\end{proof}

Since lists of primitive groups are conveniently available in computer
algebra systems such as GAP, we have checked all primitive
groups of degree at most $20$, and find no counterexamples for the statement of Theorem \ref{2.1}. In view of our
remarks about Ramsey numbers earlier, we may from now on assume that $k\ge4$.

\subsection{$2$-homogeneity}

\begin{prop}
$G$ is $2$-homogeneous.
\end{prop}

\begin{proof} Assume that $G$ is not $2$-homogeneous; let it have $r$ orbits on
$2$-element subsets. Each is the edge set of one of the symmetrised orbital
graphs for $G$; each of these graphs is vertex-primitive and edge-transitive.
First we show:
\begin{itemize}
\item each symmetrised orbital graph has valency at least $k$;
\item there are at most two such graphs (that is, $r\le2$).
\end{itemize}
For the first point, suppose that there is a graph whose valency $v-1$ is 
smaller than $k$, so that $v\le k$.
Then some $k$-set contains a vertex and all its neighbours in this graph,
and hence every $(k+1)$-set does so. The number of $(k+1)$-sets is
$n\choose k+1$, whereas the number of ways of choosing the closed
neighbourhood of a vertex in the graph, and then adjoining $k+1-v$ more
points to make a $(k+1)$-set is $n{n-v\choose k+1-v}$. Since the second
method overcounts, we have ${n\choose k+1}\le n{n-v\choose k+1-v}$. A
short calculation yields $n\le k+2$, a contradiction.

For the second, suppose that $r\ge3$, and let $\Gamma_1$, $\Gamma_2$, 
$\Gamma_3$ be three of the orbital graphs. By the first point, we can find
a $(k+1)$-set consisting of a vertex $x$ and $k$ of its neighbours in the
graph $\Gamma_1$. This set must contain a $k$-set consisting of a point
$y$ and $k-1$ of its neighbours in $\Gamma_2$, and a $k$-set consisting of
a point $z$ and $k-1$ of its neighbours in $\Gamma_3$. Now it is clear that
$x,y,z$ are distinct; but then the pair $\{y,z\}$ must be an edge in both
$\Gamma_2$ and $\Gamma_3$, a contradiction.

Now we conclude the proof. Suppose that $r=2$ and let $\Gamma$ be one of the
two complementary orbital graphs. Suppose first that the valency of $\Gamma$
is at least $k+1$. Then we can choose a $(k+2)$-set consisting of a vertex
$x$ and a set $X$ of $k+1$ of its neighbours. Now one point of $X$, say $y$,
must be joined to at least $k-1$ further points of $X$, say those in a
subset $Y$. Now the induced subgraph on the $k+1$ points $\{x,y\}\cup Y$
has minimum valency at least $2$, and so cannot contain a vertex which is
nonadjacent to all but one point of this set, a contradiction.

So $\Gamma$ has valency $k$, as does its complement, and $n=2k+1$.

There are $n\choose k+1$ choices of a $k+1$-set; each contains a vertex
joined to all or all but one of the remaining vertices. But each vertex
lies in just one $k+1$-set in which it is joined to all other vertices,
and to $k^2$ in which it is joined to all but one (choose one neighbour to
omit, and one non-neighbour to include). So ${n\choose k}\le nk^2$.

This inequality fails for $k\ge5$; if $k=4$ then $n=9$, which has already
been disposed of by computation.
\end{proof}

\subsection{$2$-transitivity}

\begin{prop}
$G$ is $2$-transitive.
\end{prop}

\begin{proof}
According to Kantor's classification, a 2-homogenous, but non 2-transitive group $G$ is contained in a
one-dimensional affine group, and has order at most $q(q-1)/2\cdot\log_pq$,
where $n=q$ is a power of $p$ and is congruent to $3$ (mod~$4$); so $p$ is
congruent to $3$ (mod~$4$) and $\log_pq$ is odd.

If $k=3$, then we have $q(q-1)/2\cdot\log_pq\ge q(q-1)(q-2)/18$, so
$q-2\le 9\log_pq$. It is easy to check that this inequality is satisfied only
for $q=7,11,27$. The first two cases are covered by computation.

Suppose that $q=27$, so that $|G|\le27\cdot13\cdot3$. If $k=4$, then our
inequality $|G|\ge{27\choose 4}/4$ is violated. As remarked earlier, this
settles larger values of $k$ also.
\end{proof}

\subsection{Completion of the proof}

We have a list of $2$-transitive groups. It is now a case of going through
the list.

The condition of $(k,k+1)$-homogeneity is closed upwards; so we can usually
assume that the groups we are considering are maximal subgroups of the
symmetric or alternating group. The only exception is when we are testing
the $(k,k+1)$-homogeneity of a group which has a $k$-homogeneous overgroup.
Since $k\ge4$ and we may assume that $n\ge20$, the only cases which need to
be considered are $\psl(2,23)\le \M_{24}$ (with $n=24$, $k=4$ or~$5$) and
$\pgl(2,32)\le\pgaml(2,32)$ (with $n=33$, $k=4$).
Computation shows that neither group is $(k,k+1)$-homogeneous.

According to Burnside, the $2$-transitive groups are of two
types: \emph{affine groups}, whose minimal normal sugroup is elementary
abelian; and \emph{almost simple groups}, whose minimal normal subgroup is
simple. For the affine groups, the maximal groups are $\agl(d,p)$ for
$p$ prime.

\paragraph{Case $G=\agl(d,p)$, with $p$ prime}

\subparagraph{Subcase $d=1$} We have $p(p-1)\ge{p\choose 4}/4$, so $p\le 11$;
these cases are excluded by computation.

So we may assume that $d\ge2$.
Below, $\mathbf{0}$ and $\mathbf{1}$ denote the all-zero and all-one vectors.

\subparagraph{Subcase $k\le d$} There is an affine independent $(k+1)$-set.
But since $k\ge4$, there is a $k$-set containing three or four affine
dependent points ($(0,\mathbf{0})$, $(1,\mathbf{0})$ and $(2,\mathbf{0})$ if
$p>2$, and $(0,0,\mathbf{0})$, $(1,0,\mathbf{0})$, $(0,1,\mathbf{0})$ and
$(1,1,\mathbf{0})$ if $p=2$) which cannot be contained in such a $(k+1)$-set.

\subparagraph{Subcase $d+1\le k\le p^{d-1}$, excluding $k=d+1$, $p=2$, $d$ odd}
If $p$ is odd, or if $p=2$ and $d$ is odd,
there exist $d+2$ points such that every hyperplane omits at least two,
namely $\mathbf{0}$, $\mathbf{1}$, and the points with a single coordinate $1$
and all others zero. (This construction needs to be modified if $p$ is odd
and divides $d-1$: then replace the all-$1$ vector by $(2,\mathbf{1})$.) 
A $(k+1)$-set containing it can contain no $k$ points contained in a 
hyperplane, a contradiction.

If $p=2$ and $d$ is odd, then $d\ge5$, so we can add one more point and 
find a set of size $d+3$ with the claimed property: any non-zero vector with
an even number of $1$s will do.

\subparagraph{Subcase $k=d+1$, $d$ odd, $p=2$} Since $d+2\le2^{d-2}$, we
can take a $(k+1)$-set contained in a $(d-2)$-flat and an affine independent
$k$-set.

\subparagraph{Subcase $p^{d-1}-1\le k\le p^d-d(p-1)-2$} There is a set of
$1+d(p-1)$ points meeting every hyperplane, namely those with at most one
non-zero coordinate. Its complement contains a $(k+1)$-set omitting a point
of every hyperplane. But there is a $k$-set containing a hyperplane.

\subparagraph{Subcase $k\ge p^d-d(p-1)-1$} Since $2k+1\le p^d$, we have
$p^d\le2d(p-1)+1$, which is satisfied only for $p=3$, $d=2$, giving the
known examples.

\medskip

For groups with simple socle, there are more cases.

\paragraph{Case: $G=\pgaml(2,q)$, $q=p^e$}

\subparagraph{Subcase $k\ge5$}
We have $|G|=(q+1)q(q-1)e$. If $k\ge5$, then $|G|\ge{q+1\choose 5}/5$, so
$(q-2)(q-3)\le600e$, so $q\le27$ or $q=32$, handled by computation. 
For $k\ge6$ the inequality gives $q\le17$, which is covered by computation.

\subparagraph{Subcase $k=4$}
The orbits of $\pgl(2,q)$ on the $4$-tuples of distinct points are parametrised
by cross-ratio, of which there are $q-2$ distinct values. A typical $4$-set
has six distinct cross-ratios; depending on the congruence of $q$, there may
be a set with only two cross-ratios, and one with only three. So the group
$\pgl(2,q)$ has at least $(q+5)/6$ orbits on $4$-sets. Adding field 
automorphisms at worst divides the number of orbits by $e$. So
$q+5\le24e$. If $e=1$ then $q\le19$, covered by our computation. For $e>1$,
the remaining values to be checked are $q=25$, $27$, $32$, $64$ and $128$;
again computation shows there are no examples.

\paragraph{Case: $G$ is a unitary, Suzuki or Ree group}

These groups are smaller than $\psl(2,q)$ of the same degree; all are ruled
out by the order test except for $\pgamu(3,q)$ with
$q=3,4$. Now $\pgamu(3,3)$ (with degree $28$) is a
subgroup of $\Sp(6,2)$, considered below. $\pgamu(3,4)$
is handled by computation.

\paragraph{Case $G=\pgaml(d,q)$, $d\ge3$}
Here $n=(q^d-1)/(q-1)$. We follow similar arguments to the affine case.

\subparagraph{Subcase: $k\le d$} There exist $d+1$ points, no three
collinear. On the other hand, there is a $k$-set containing three
collinear points.

\subparagraph{Subcase: $d<k<(q^{d-1})/(q-1)$} There is a $k$-set containing a basis
for the vector space, and a $(k+1)$-set contained in a hyperplane.

\subparagraph{Subcase: $(q^{d-1}-1)/(q-1)\le k<n-(q+1)$} Since a line contains $q+1$
points and meets every hyperplane, there is a $(k+1)$-set containing no
hyperplane; but there is a $k$-set which contains a hyperplane.

\subparagraph{Subcase: $k\ge n-(q+1)$} In this case, $k>n/2$, contrary to assumption.

\paragraph{Case: $G=\Sp(2d,2)$, with $n=2^{2d-1}\pm2^{d-1}$}

We start with a brief description of these groups. Let $V$ be a vector space
of dimension $2d$ over the field of two elements, and $B$ a fixed nondegenerate
alternating bilinear form on $V$. Let $\mathcal{Q}$ be the set of all quadratic
forms on $V$ which polarize to $B$. These fall into two orbits under the
action of the symplectic group $\Sp(d,2)$, of sizes
$2^{2d-1}\pm2^{d-1}$,
corresponding to the two types of quadratic form, distinguished by the
dimension of their maximal totally singular subspaces. The two types are
designated $+$ and $-$, and the corresponding dimensions are $d$ and $d-1$
respectively. Let $\mathcal{Q}_\epsilon$ be the set of forms of type $\epsilon$.
The symplectic group is $2$-transitive on each orbit. We may assume that
$d\ge3$, since otherwise the degrees are smaller than $20$.

It is readily checked from the order bound that the values of $k$ which need
to be considered satisfy $k+1\le|W|$, where $W$ is a maximal totally singular
subspace of the relevant quadratic form, except for $d=3$ and type $-$ (acting
on $28$ points). This exceptional case can be handled by computation.

There is a ternary relation on $\mathcal{Q}_\epsilon$ preserved by the group.
If $Q_1,Q_2,Q_3$ are three quadratic forms of the same type, then 
$Q_1+Q_2+Q_3$ is a quadratic form, which may be of the same or opposite type.
Let $\mathcal{T}$ be the set of all triples for which the sum is of the same
type. It is easily checked that, for a fixed form $Q$, and a maximal totally
singular subspace $W$ for $Q$, the set $\{Q_w:w\in W\}$, where
$Q_w(x)=Q(x)+B(x,w)$, is a set of $|W|$ forms, all triples of which belong
to $\mathcal{T}$. Since $k+1\le|W|$, and there exists a triple not belonging
to $\mathcal{T}$, we see that $G$ cannot be $(k,k+1)$-homogeneous.

\paragraph{Case: $G$ is sporadic}

The sporadic $2$-transitive groups of degree greater than $20$ are $\M_{22}$
and its automorphism group ($n=22$), $\M_{23}$ ($n=23$), $\M_{24}$ ($n=24$),
the Higman--Sims group ($n=176$) and the Conway group $\Co_3$ ($n=276$).
Computation handles all of these except the Conway group, which is a bit
on the large side. However, the order test shows that we only need consider
$k\le6$; the case $k=4$ yields to computation, and the other cases cannot arise
since inspection of the combinatorial object preserved by the group (a
so-called ``regular two-graph'' see \cite{taylor}) shows that there are seven substructures on
five points, and so at least seven orbits on $5$-sets and on $6$-sets.

\section{The analogue of the Livingstone--Wagner result}

Livingstone and Wagner \cite{lw} proved that if a group $G\leq S_{n}$ is $k$-homogeneous (for $2\leq k\leq \lfloor \frac{n}{2}\rfloor$), it is also $(k-1)$-homogeneous.

\begin{cor}\label{cor2b}
Let $n\geq 1$, let $3\leq k\leq \lfloor \frac{n+1}{2}\rfloor$, and let $G\leq \sym$ be a $(k-1,k)$-homogeneous group. Then $G$ is a $(k-2,k-1)$-homogeneous group, except when  $n=9$ and $G\cong\asl(2,3)$ or $\agl(2,3)$, with $k=5$.  
\end{cor}
\begin{proof}
We know that $G$ either is $(k-1)$-homogeneous or is one of the five exceptions listed in Theorem \ref{2.1}. If the group is $(k-1)$-homogenous, then it certainly is $(k-2,k-1)$-homogeneous. 

Regarding the three exceptions of degree 5 and 7, GAP shows they satisfy the corollary. Regarding the two groups of degree 9, GAP shows that they are not $(3,4)$-homogeneous as for both groups the orbit of $\{1,2,3\}$ does not contain a subset of $\{1,2,4,5\}$.   
\end{proof}

Recall that if $G$ is a permutation group of degree~$n$, for $k\leq n$, we say that
$G$ possesses the $k$-universal transversal property if the orbit of any $k$-set contains a section for every $k$ partition of $[n]$.  It is clear that the class of groups the $k$-universal transversal property (for some $2\leq k\leq \lfloor n/2\rfloor$)  is contained in the class of  $(k-1,k)$-homogeneous groups (consider a $k$-partition with $k-1$ singleton blocks); therefore the groups possessing the $k$-universal transversal property are $(k-1)$-homogenous, with the exceptions listed in Theorem \ref{2.1}, and hence they have the $(k-1)$-universal transversal property, with the possible exception of the five exceptional groups listed in Theorem \ref{2.1}. Inspection of these groups leads to the following result.

\begin{cor}\label{cor1b}
Let $n\geq 5$, let $2\leq k\leq \lfloor \frac{n+1}{2}\rfloor$, and let $G\leq \sym$ be a group having the $k$-universal transversal property. Then $G$ has the $(k-1)$-universal transversal property.  
\end{cor}

If $k>\lfloor\frac{n+1}{2}\rfloor$, then no analogue of the Livingstone--Wagner Theorem can hold. But the situation is actually very simple.

\begin{thm}\label{bigk}
Let $G$ be a subgroup of $\sym$, and let $k$ be an integer satisfying
$\lfloor\frac{n+1}{2}\rfloor < k < n$. Then the following are equivalent:
\begin{enumerate}
\item $G$ has the $k$-universal transversal property;
\item $G$ is $(k-1,k)$-homogeneous;
\item $G$ is $k$-homogeneous.
\end{enumerate}
In particular, if these condiditions hold with $k<n-5$, then $G$ is $\sym$
or $\alt$.
\end{thm}

\begin{proof}
(1) implies (2): The proof of this given earlier does not depend on the
value of $k$.

(2) implies (3): Let $G$ be $(k-1,k)$-homogeneous. Then $G$ is
$(n-k,n-k+1)$-homogeneous. Because of the inequality on $k$, the exceptional
groups in Theorem~\ref{th6} do not occur; so $G$ is $(n-k)$-homogeneous, and
hence $k$-homogeneous.

(3) implies (1): Clear.
\end{proof}

\section{The classification of the groups possessing the $k$-universal transversal property}\label{kutp}

Let $G$ be a permutation group of degree~$n$. For $k\leq n$, we say that
$G$ possesses the $k$-universal transversal property if the orbit of any $k$-set contains a section for every $k$ partition of $[n]$. A group is said to have the universal transversal property if it has the $k$-universal transversal property for all $k\leq n$. In \cite{ArMiSc} the following theorem is  proved.

\begin{thm} \label{superlemma}
A subgroup $G$ of $\sym$ has the universal
transversal property  if and only if one of the following is valid:
\begin{enumerate}
    \item[\rm (i)] $n=5$ and $G\cong C_5,\ \dihed{5},$ or $\agl(1,5)$;
    \item[\rm (ii)] $n=6$ and $G\cong \psl(2,5)$ or $\pgl(2,5)$;
    \item[\rm (iii)]  $n=7$ and $G\cong\agl(1,7)$;
    \item[\rm (iv)] $n=8$ and $G\cong\pgl(2,7)$;
    \item[\rm (v)] $n=9$ and $G\cong\psl(2,8)$ or $\pgaml(2,8)$;
    \item[\rm (vi)] $G=\alt$ or $\sym$.
\end{enumerate}
\end{thm}

The goal of the following two sections is to prove that with some exceptions, the groups possessing the $k$-universal transversal property are $k$-homogeneous (for $2\leq k\leq \lfloor \frac{n}{2}\rfloor$).
We abbreviate ``$k$-universal transversal property'' to
\textit{$k$-ut property}.

Our main results are Theorems \ref{th4a}--\ref{th4f} stated in the introduction, and that we now state in a single theorem.

\begin{thm}\label{RegularAndKut}
Let $n\geq 11$, $G\leq \sym$ and let $2\leq k\leq \lfloor \frac{n}{2}\rfloor$.
\begin{enumerate}
	\item If $6\leq k\leq \lfloor \frac{n}{2}\rfloor$, then the following are equivalent: 
		\begin{enumerate}
			\item $G$ has the $k$-ut property; 
			\item $\alt\leq G$. 
		\end{enumerate}
		\vspace{0.3cm}
	\item The following are equivalent: 
		\begin{enumerate}
			\item $G$ has the $5$-ut property; 
			\item $G$ is 5-homogeneous, or  $n=33$ and $G=\pgaml(2,32)$.
		\end{enumerate}
\vspace{0.3cm}
\item  
		\begin{enumerate}
			\item if $G$ is $4$-homogenous,  or $n=12$ and $G=\M_{11}$, then $G$ has the $4$-ut property. 
			\item Apart from the possible exception of some $G$ with $\psl(2,q)\leq G\leq \pgaml(2,q)$,   				the 	groups listed above are the only ones having the $4$-ut property. 
		\end{enumerate}
\vspace{0.3cm}
\item 
	\begin{enumerate}
			\item if $G$ is $3$-homogenous,  or one of the following groups 
					\begin{enumerate}
						\item $\psl(2,q)\leq G\leq \psigl(2,q)$, where $q\equiv 1 (\mbox{mod }4)$; 
						\item $\Sp(2d,2)$ with $d\geq 3$, in either of its $2$-transitive representations;
						\item $2^{{2d}}:\Sp(2d,2)$;
						\item Higman--Sims;
						\item $\Co_{3}$;
						\item $2^{{6}}:\G_{2}(2)$ and its subgroup of index $2$;
						\item $\agl(1,p)$ where, for all $c\in \GF(p)^{{*}}$, $|\langle -1,c,c-1\rangle|=p-1$; 
					\end{enumerate}
					 then $G$ has the $3$-ut property. 
			\item Apart from the possible exception of the Suzuki groups $\Sz(q)$,   				the 			groups listed above are the only ones having the 3-ut property. 
		\end{enumerate}
		\vspace{0.3cm}
\item 	The following are equivalent. 
\begin{enumerate}
			\item $G$ has the $2$-ut property; 
			\item $G$ is primitive. 
		\end{enumerate}
\end{enumerate}
For $n<11$, a group  $G\leq \sym$ with the $k$-universal transversal property is
$k$-homogeneous, with the following exceptions: 
\begin{enumerate}
\item[(i)] $n=5$,  $C_{5}$ or $D(2*5)$ and $k=2$;
\item[(ii)] $n=6$, $\psl(2,5)$ and $k=3$;
\item[(iii)]  $n=7$,  $C_{7}$ or $D(2*7)$, and $k=2$; or 
$\agl(1,7)$ and $k=3$;
\item[(iv)] $n=8$, $\pgl(2,7)$ and $k=4$;
\item[(v)]  $n=9$,  $3^{{2}}:4$ or $3^{{2}}:D(2*4)$ and $k=2$;
\item[(vi)] $n=10$, $\mathcal{A}_{5} $ or $\mathcal{S}_{5}$ and $k=2$; or 
$\psl(2,9)$ or $\mathcal{S}_{6}$   and $k=3$.
\end{enumerate}
\end{thm}

\begin{proof}
In Sections \ref{kutp} and \ref{exceptional} all the claims for $n\geq 11$ are proved. Regarding $n<11$, all the claims can be easily checked with GAP. 

For $n=5,6$ all the possible groups appear in the statement of the theorem.  

For $n=7$ the group $7:3$ is $2$-homogenous, but does not have the $3$-ut property as the orbit of $\{1,2,7\}$ has no transversal  for $\{ \{ 1 \}, \{2, 7 \}, \{ 3, 6, 4, 5 \} \}$. The $2$-homogeneous group $L(3,2)$ also does not have the $3$-ut property as the orbit of $\{1,2,4\}$ does not contain a transversal for $\{ \{ 1 \}, \{ 2, 4 \}, \{ 3, 7, 6, 5 \} \}$.

For $n=8$, the $3$-homogenous groups $\agl(1,8)$, $\agaml(1,8)$, $\asl(3,2)$ do not have the $4$-ut property. There is no section for the partition $\{ \{ 1 \}, \{ 2 \}, \{ 3, 4 \}, \{ 5, 6, 7, 8 \} \}$ in the  orbits of the set $\{1,2,3,4\}$. 
The $3$-homogenous group $\psl(2,7)$  does not have the $4$-ut property as the orbit of $\{1,2,3,5\}$ has no transversal for the partition $\{\{ 1, 5, 7, 8 \}, \{ 2, 6 \}, \{ 3 \}, \{ 4 \} \}$.

For $n=9$, the $2$-homogenous groups $\M_{9}$, $\agl(1,9)$, $\agaml(1,9)$, $\asl(2,3)$, $\agl(2,3)$ do not have the $3$-ut property. Their orbits on the set $\{1,2,3\}$ contain no section for the partition $\{ \{ 1 \}, \{ 2,3 \}, \{ 4 , 5, 6, 7, 8,9 \} \}$.

For $n=10$ the $3$-homogeneous groups $\pgl(2,9)$, $\M_{10}$, $\pgaml(2,9)$   do not have the $4$-ut property as the partition   $\{ \{ 1 \}, \{ 2 \}, \{ 3, 10 \}, \{ 4, 9, 8, 6, 7, 5 \} \}$ has no transversal in the orbit of $\{1,2,3,10\}$.
\end{proof}

\begin{prop}\label{critical}
\begin{enumerate}
\item A $k$-homogeneous group has the $k$-ut property.
\item A group with the $k$-ut property is $(k-1,k)$-homogeneous, and hence
is $(k-1)$-homogeneous or one of the five exceptions to Theorem~\ref{th6}.
\item If $k>n/2$, then a group has the $k$-ut property if and only if it is $k$-homogeneous
or one of the exceptions to Theorem~\ref{th6}.
\end{enumerate}
\end{prop}

The first part is trivial; the second is contained in the preamble to
Theorem~\ref{th5} and the result of Theorem~\ref{th6}; the third is contained
in the preamble to Theorem~\ref{kkp1}.

Our aim is to determine, as completely as possible, the groups with the
$k$-ut property which are not $k$-homogeneous. For $k=2$, there are many
such groups and no hope of a determination:

\begin{prop}
A subgroup of $\sym$ has the $2$-ut property if and only if it is primitive.
\label{primlem}
\end{prop}

\begin{proof}
By Higman's Theorem~\cite{higman}, $G$ is primitive if and only if all the
non-diagonal orbital graphs are connected. (These are just the graphs whose
edge sets are the orbits on $2$-sets.) But a graph is connected if and
only if, for every 2-partition of the vertices, there is an edge which is
a section for the partition.
\end{proof}

However, for $k>2$, we are in a stronger position, due to the following pair
of results, one negative, one positive:

\begin{prop}
If $G$ is a group of automorphisms of a Steiner system $S(k-1,l,n)$
with $k-1<l<n$, then $G$ does not have the $k$-ut property.
\end{prop}

(A Steiner system is a collection of blocks or subsets of size $l$ of the
$n$-set so that each $(k-1)$-set is contained in a unique block.)

\begin{proof}
Take the partition with $k-2$ singleton parts, one part of size
$l-k+2$ consisting of the remaining points in some block containing these
points, and one part consisting of everything else. A $k$-set which is
contained in a block cannot be a section for this partition.
\end{proof}

This shows that the following $(k-1)$-homogeneous groups do not have the $k$-ut property:
\begin{enumerate}
\item Subgroups of $\agl(d,p)$, if  $d>1$ and $p>2$, with $k=3$ (these
groups preserve the geometry of affine points and lines).
\item $3$-transitive subgroups of $\agl(d,2)$ for $d>2$, with $k=4$ (these
preserve geometry of affine points and planes).
\item Subgroups of $\pgaml(d,q)$ with $d>2$, for $k=3$ (these preserve the
geometry of projective points and lines).
\item The unitary and Ree groups, with $k=3$ (these preserve unitals).
\item Subgroups of $\pgaml(2,q)$ containing $\psl(2,q)$, where $q$ is a proper
power of an odd prime or $q=2^e$ where $e$ is not prime, with $k=4$ (these
preserve circle geometries). Note that we can exclude subgroups of
$\psigl(2,q)$ containing $\psl(2,q)$ for $q\equiv1$ (mod~$4$), since these
groups are not $3$-homogeneous.
\item The Mathieu groups $\M_{11}$, $\M_{12}$, $\M_{22}$ (and $\aut(\M_{22})$),
$\M_{23}$ and $\M_{24}$ in their usual representations, with $k=5,6,4,5,6$,
respectively, as these preserve famous Steiner systems.
\end{enumerate}

A permutation group $G$ is \emph{$k$-primitive} if it is $(k-1)$-transitive
and the pointwise stabiliser of $k-1$ points acts primitively on the
remaining points. It is \emph{generously $k$-transitive} if the setwise
stabiliser of $k+1$ points induces the symmetric group on these points, and
is \emph{almost generously $k$-transitive} if the setwise stabiliser of
$k+1$ points induces the symmetric or alternating group on them 
(Neumann~\cite{pmn:generous}).

\begin{lem}
If a permutation group $G$ is $k$-primitive and almost generously
$k$-transitive, then every orbit of $G$ on $(k+1)$-sets contains a section 
for any $(k+1)$-partition of which $k-1$ of the parts are singletons.
\end{lem}

\begin{proof}
In an almost generously $k$-transitive group, each orbit on $(k+1)$-sets
corresponds in a natural way to an orbit of the $(k-1)$-point stabiliser on
pairs of points outside the given $k-1$ points. Now the proof concludes as
in Proposition~\ref{primlem}.
\end{proof}

\begin{prop}\label{two-graph-prop}
Each of the following $2$-transitive groups $G$ has the $3$-ut property:
\begin{enumerate}
\item $\psl(2,q)\le G\le\psigl(2,q)$, where $q\equiv1$ ($\mathrm{mod}\;4$);
\item $\Sp(2d,2)$ with $d\ge3$, in either of its $2$-transitive representations;
\item $2^{2d}\colon\Sp(2d,2)$;
\item $\Co_3$.
\end{enumerate}
\end{prop}

\begin{proof}
Each of these groups is $2$-primitive and generously $2$-transitive, so by
the previous lemma it is enough to consider partitions $\{X,Y,Z\}$ in which no part is
a singleton. We assume that $X$ is the smallest part.

Also, each of these groups has just two orbits on $3$-sets, and each orbit $T$
is a \emph{regular two-graph} (Taylor~\cite{taylor}), that is,
\begin{itemize}
\item any two points lie in exactly $\lambda$ members of $T$, for some
$\lambda$;
\item any four points contain an even number of members of $T$.
\end{itemize}

Now suppose that the orbit $T$ contains no section for the partition
$\{X,Y,Z\}$. For any $x_1,x_2\in X$,
$y\in Y$ and $z\in Z$, we have $x_1yz, x_2yz\notin T$, and so both or neither
of $x_1x_2y$ and $x_1x_2z$ belong to $T$. Suppose that $x_1x_2y\in T$, for
some $y\in Y$. Then $x_1x_2z\in T$ for all $z\in Z$, and $x_1x_2y\in T$ for all
$y\in Y$. Hence we have $\lambda \ge |Y|+|Z| \ge 2n/3$. In the contrary case,
neither of these triples belong to $T$, and $\lambda \le |X|-2 \le n/3 - 2$.

However, for these groups, it is easily checked that these inequalities fail
in all cases:
\begin{itemize}
\item for $\psl(2,q)$, $n=q+1$, $\lambda=(q-1)/2$;
\item for $\Sp(2d,2)$, $n = 2^{2d-1} \pm 2^{d-1}$, $\lambda=2^{2d-2}$ or
$\lambda=2^{2d-2}\pm 2^{d-1}-2$;
\item for $2^{2d} \colon \Sp(2d,2)$, $n = 2^{2d}$, $\lambda = 2^{2d-1}$ or
$\lambda=2^{2d-1}-2$;
\item for $\Co_3$, $n=276$, $\lambda=112$ or $\lambda=162$.
\end{itemize}
\end{proof}

These results, together with computation for $n\le 11$, resolve
the question of the $k$-ut property for all groups $G$ which
are $(k-1)$-homogeneous but not $k$-homogeneous, but for the following exceptional cases:

\begin{enumerate}
\item $k=3$:
  \begin{enumerate}
  \item $\agl(1,p)$, $p$ prime, or its subgroup of index $2$ if $p \equiv 3$
 (mod~$4$);
  \item $2^6 \colon \G_2(2)$ and its subgroup of index $2$;
  \item $\Sz(q)$;
  \item The Higman--Sims group.
  \end{enumerate}
\item $k=4$:
  \begin{enumerate}
  \item $\psl(2,q) \le G \le\pgaml(2,q)$, with either $q$ prime (except 
  $\psl(2,q)$ for $q\equiv1$ (mod~$4$), which is not $3$-homogeneous), or $q=2^p$ for $p$ prime;
  \item $\M_{11}$, degree $12$.
  \end{enumerate}
\item $k=5$:
  \begin{enumerate}
  \item $\pgaml(2,32)$, degree $33$.
  \end{enumerate}
\end{enumerate}

\section{The Exceptional Cases}\label{exceptional}

In this section we are going to look at the exceptional  cases listed at the end of the previous section.

\subsection{The group $\M_{11}$ \ut{4}}

The group $\M_{{11}}$ of degree 12 can easily be handled with GAP.  This group has two orbits on $4$-sets (in GAP are the orbits of $\{1,2,3,4\}$ and of $\{1,2,3,7\}$). GAP checks in less than one minute that the orbit of each one of these sets contains a section for all possible $4$-partitions of $\{1,\ldots,12\}$.

\subsection{Some general results}\label{general}

In this subsection we are going to prove a number of auxiliary results.  We start by associating a graph to a $t$-homogeneous group $G\leq \sym$ as follows. Let $B\subseteq [n]$ and $c\in [n]$ such that $|B|=t-1$. Then we define the following graph on the points in $[n]\setminus B$:
\[
G(B,c)=\{ \{x,y\}  \mid  \{x,y\}\cup B \in (\{1,\ldots ,t,c\})G \}.
\]   
In the particular case of $B=\{b\}$, we will write $G(b,c)$ rather than $G(\{b\},c)$.  

Observe that for every $g\in G$ we have $G(B,c)\cong G(Bg,c)$. Therefore, as $t$ is larger than the order of $B$, it follows that $G(B,c)$ is connected if and only if $G(B',c)$ is connected, for all $B'$ such that $|B'|=|B|=t-1$. 

\begin{prop}\label{connected} 
If a $t$-homogeneous group $G\leq \sym$ \ut{(t+1)}, then $G(B,c)$ is connected, for all $B\subset [n]$ (with $|B|=t-1$) and all $c\in [n]$.
\end{prop}
\begin{proof}
If $G(B,c)$ is not connected, then there exists 
a connected component $D$ contained in the graph. Now consider the partition $P=(\{b_{1}\},\dots,\{b_{t-1}\},D,R)$, where $R$ contains the remaining elements, that is, $R=[n]\setminus (B\cup D)$, and $B=\{b_{1},\ldots,b_{t-1}\}$. 

Any section for the partition $P$ must contain $B$. Therefore any set containing $B$ and in the orbit of  $\{1,\ldots ,t,c\}$ must be of the form $\{x,y\}\cup B$ and hence $x$ and $y$ are connected in $G(B,c)$. Thus, either $x\in D$ and hence $y\in D$ (because $D$ is a connected component of $G(B,c)$) so that     $\{x,y\}\cup B$ is not a section for $P$; or $x\not\in D$ and hence $y\not \in D$ thus implying $x,y\in R$. Again $\{x,y\}\cup B$ is not a section for $P$. The result follows. 
\end{proof}

This proposition immediately implies that, for example, the 2-homogeneous group $G=\agl(1,17)$ \notut{3}. 

In fact, according to GAP, the graph $G(\{17\},4)$ has the following two connected components 
$\{ \{1, 2, 3, 4, 6, 9, 10, 15 \}, \{ 5, 7, 8, 11, 12, 13, 14, 16 \}\}$. And it can be checked, in fact, that the orbit of $\{1,2,4\}$ under $G$ has no section for the partition $$P=( \{1, 2, 3, 4, 6, 9, 10, 15 \}, \{ 5, 7, 8, 11, 12, 13, 14, 16 \},\{17\}).$$ (More on these groups below.)

For the particular case of the $3$-ut property, another important graph is the following: for a set $C\subseteq [n]$ and $c\in [n]$, we have 
\[
\Gamma(C,c)=\bigcup_{{b\in C}}\left(G(b,c)\cap \left[([n]\setminus C)\times([n]\setminus C)\right] \right).
\]

Let $G$ be a group admitting a  bad 3-partition, that is, $P=(A,C,A')$ such that no set in the orbit of $\{1,2,c\}$ is a section for $P$. This means that the distance from $A$ to $A'$ in the graph $\Gamma(C,c)$ must be infinite. In fact, if it is not infinite, it must be one as every vertex in this graph either is on $A$ or in $A'$. That means that there exist $a\in A$ and $a'\in A'$ that are connected in $\Gamma(C,c)$. But, by definition, $\Gamma$ is a union of subgraphs of $G(b,c)$ and hence it follows that for some $b\in C$ we have that $\{a,a'\}$ is an edge in $G(b,c)$. Thus $b\in C$, $a\in A$, $a'\in A'$ and hence $\{a,b,a'\}$ is a section for $P$ that belongs to the orbit of $\{1,2,c\}$, by the definition of $G(b,c)$.   		  It is proved that bad partitions $P=(A,C,A')$ induce graphs   $\Gamma(C,c)$ in which the distance from $A$ to $A'$ is infinite. 

This observation leads to the following procedure that (if it ends) allows to check that a group $G$ \ut{3}. We already know that if $G(n,c)$ is disconnected, then the group \notut{3}. The question is whether there exists a group $G$ with connected graph $G(n,c)$, but that \notut{k}.  Therefore we assume that $G(n,c)$ is connected and start with three sets $(A,C,A')$, all contained in $[n]$, such that (without loss of generality because we only consider 2-homogeneous transitive groups) $\{1\}= A$ and $\{n\}= C$, where $n$ is the degree of the group $G$. We are going to try to build a bad partition and hence include in $A\cup A'$ and $C$ all the elements that must necessarily be in each one of this sets provided that we want $\Gamma(C,c)$ to be disconnected on $A\cup A'$. So we proceed as follows (denote the distance from $a$ to $b$ in graph $Gr$ by $D_{{Gr}}(a,b)$): for a fixed $d\in [n]$ that will be the distance in $G(n,c)$ from $A$ to $A'$,

\begin{enumerate}
\item put in $A'$ an element $y\in [n]$ such that $D_{G(n,c)}(1,y)=d$; add to $A\cup A'$ the set $\{x\in [n] \mid \{1,y\}\in G(x,c)\}$. (Observe that these $x$ must be in $A\cup A'$ because if one of them is  in $C$, then $\{1,y,x\}$ would be a section for $(A,C,A')$ and hence any oversets of $A,A',C$ yielding a partition of $[n]$ would have a section in the orbit of $\{1,2,c\}$.)  
\item add to $C$ the set $\{x\in [n]\mid 1-\ldots-x-\ldots -y\}$, where $a-b$ means that $\{a,b\}$ is an edge in $G(n,c)$. (Observe that if one of these $x$ is not in $C$, then $D_{G(n,c)}(A,B)<d$, contrary to our assumption.)
\item check if $A\cup A'$ is connected under $\Gamma(C,c)$; if it is connected, then all the overpartions of $(A,C,A')$ are good; if it is not connected, then
\item add to $A\cup A'$ the set $\{x\in [n]\mid A\cup A' \mbox{ is connected in } \Gamma(C\cup \{x\},n)\}$ (as such $x$ cannot go to $C$). 
\item $A\cup C\supseteq \{x\mid D_{G(n,c)}(A,x)<d\}$ and $A'\cup C\supseteq \{x\mid D_{G(n,c)}(A',x)<d\}$; 
\item add to $A$ the set $(A\cup C)\cap (A\cup A')$.
\item add to $A'$ the set $(A'\cup C)\cap (A\cup A')$.
\item add to $C$ the set $(A\cup C)\cap (C\cup A')$.
\item go to (3). 
\end{enumerate}

\subsection{The two exceptional groups of degree 64}\label{64}
The group $H=2^{{6}} :\G_{2}(2)$ and its subgroup $G$, of index 2, have $G(64,c)$ connected. Therefore the guess is that they have the  $3$-ut property. To test that conjecture we are going to apply to $G$ the procedure outlined at the end of the previous subsection, as if $G$ \ut{3}, then the overgroup $H$ also has. On the set   $[64]$ the group $G$ has three orbits on $3$-sets, namely (in GAP) $\{1,2,3\}$, $\{1,2,5\}$ and $\{1,2,29\}$.

The points $y$ such that  $D_{G(64,3)}(1,y)=2$ are 
\begingroup
\everymath{\scriptstyle}
\tiny
$$\{2, 3, 4, 13, 14, 15, 16, 21, 22, 23, 24, 25, 26, 27, 28, 37, 38, 39, 40, 41, 42, 43, 44, 49, 50, 51, 52, 61, 62, 63\}.$$ 
\endgroup
To all of them the procedure ends yielding the result that at a certain point $A$ and $A'$ are connected under $\Gamma(C,3)$. And there are no $y$ such that  $D_{G(64,3)}(1,y)>2$.

The points $y$ such that  $D_{G(64,5)}(1,y)=2$ are 
\begingroup
\everymath{\scriptstyle}
\tiny
$$\begin{array}{c}\{5, 6, 7, 8, 9, 10, 11, 12, 15, 17, 18, 19, 20, 23, 25, 29, 30, 31, 32, 33, 34, \\ 35, 36, 40, 42, 45, 46, 47, 48, 50, 53, 54, 55, 56, 57, 58, 59, 60 \}.\end{array}$$ 
\endgroup
To all of them the procedure ends yielding the result that at a certain point $A$ and $A'$ are connected under $\Gamma(C,5)$. And there are no $y$ such that  $D_{G(64,5)}(1,y)>2$.

 Regarding $\{1,2,29\}$ (the $3$-set with the smallest orbit), the points $y$ such that  $D_{G(64,29)}(1,y)=2$ are 
\begingroup
\everymath{\scriptstyle}
\tiny
$$\{2, 3, 4, 13, 14, 16, 21, 22, 24, 26, 27, 28, 37, 38, 39, 41, 43, 44, 49, 51, 52, 61, 62, 63 \}.$$ 
\endgroup
To all of them the procedure ends yielding the result that at a certain point $A$ and $A'$ are connected under $\Gamma(C,29)$. In this case  there are also some $y$ such that  $D_{G(64,29)}(1,y)=3$:
\begingroup
\everymath{\scriptstyle}
\tiny
$$\{  5, 6, 7, 8, 9, 10, 11, 12, 17, 18, 19, 20, 29, 30, 31, 32, 33, 34, 35, 36, 45, 46, 47, 48, 53, 54, 55, 56, 57, 58, 59, 60  \}.$$
\endgroup
To all of them the procedure ends yielding the result that at a certain point $A$ and $A'$ are connected under $\Gamma(C,29)$. And there are no $y$ such that  $D_{G(64,29)}(1,y)>3$.

It is checked that $G$ \ut{3} and hence  the same holds for $H$.

      \subsection{The group $\mbox{P}\Gamma\mbox{L}(2,32)$}\label{pgaml}

The graph $G(\{1,2,3\},c)$ is connected and hence the guess is that this group \ut{5}. 
It is too big to be tested directly and the algorithm used in subsection \ref{64} does not work here. Therefore we used the following algorithm (in GAP) to prove that indeed this group \ut{5}. 

$G$ has three orbits on 5-sets: $\{1,\ldots,5\}$, $\{1,\ldots,4,6\}$ and $\{1,\ldots,4,10\}$ are representatives. Suppose we want to prove that $\{1,\ldots,4,10\}G$ contains a section for all partitions. To do that we are going to try to build a bad partition (one that has no section in  $\{1,\ldots,4,10\}G$).
\begin{enumerate}
\item We start with the subpartition $p:=\{\{1\},\{2\},\{3\},\{4\},\{6\}\}$ and add the number $5$  (the smallest not in $\{1,2,3,4,6\}$) to the blocks of $p$ in the five possible ways. We get $5$ subpartions and remove all the partitions that have a section in $\{1,\ldots,4,10\}G$. According to GAP all the $5$ are left.
\item Repeat the previous step with $7$ being included in each one of the previous five subpartitions, and then removing the ones that have a section in $\{1,\ldots,4,10\}G$. According to GAP all the $25$ are left.
\item Repeating with $8$, we end up with $115$ subpartitions. 
\item With $9$, we get $337$ subpartitions. 
\item Repeating with $10,\ldots,19$, we get, respectively, the following number of subpartitions $$719,1052,1065,1193,912,229,211,137,50,2.$$ 
 And it does not matter where we put $20$, we end always with a partition admitting a section in $\{1,\ldots,4,10\}G$. This proves that $\{1,\ldots,4,10\}G$ contains a section for all the partitions in which $1,2,3,4,6$ are all in different blocks. 
\item  Then we start with a subpartition $\{\{1\},\{2\},\{3\},\{4\},\{5\}\}$ and the orbit $\{1,\ldots,4,10\}G$, and follow the previous algorithm. Again we get that $\{1,\ldots,4,10\}G$ contains a section for all 5-partitions (in which $1,2,3,4,5$ are in different blocks). This proves that $\{1,\ldots,4,10\}G$ contains a section for all the 5-partitions since there is only another orbit of partitions: those in which $1,2,3,4,10$ are in different blocks and for those $\{1,\ldots,4,10\}G$ trivially contains a section.
\item Finally we repeat the same algorithm with $\{1,\ldots,5\}G$ and $\{1,\ldots,4,6\}G$. The worse case is when $P:=\{\{1\},\{2\},\{3\},\{4\},\{6\}\}$ and we have the orbit of $\{1,\ldots,4,10\}$.
\end{enumerate}

\subsection{The Higman--Sims group }

Let $G$ be the Higman--Sims group, a group of order $2^{{9}}\cdot 3^{{2}}\cdot 5^{{3}}\cdot 7\cdot 11$ and degree $176$.
Then $G$ has three orbits $O_1$, $O_2$, $O_3$ on ordered triples, with cardinalities  $176\cdot175\cdot t$, for $t=12,72,90$; representatives of the orbits are $(1,2,16)$, $(1,2,3)$ and $(1,2,6)$ respectively.

According to Taylor~\cite{taylor}, $O_2$ is a regular two-graph with $\lambda=72$, and exactly the same argument as in Proposition~\ref{two-graph-prop} shows that every $3$-partition has a section belonging to the orbit $O_2$. So we only have to deal with the orbits $O_1$ and $O_3$.

Moreover, every $3$-partition is equivalent under $G$ to one with $1,2,3$ in different parts; so we started with the subpartition $\{\{1\},\{2\},\{3\}\}$ and applied the same algorithm used in the preceding subsection, concluding that this group has the $3$-ut property. 

  \subsection{The groups $\agl(1,p)$ for $p$ prime}\label{numbertheory}
  
  We proved above that if the graph $G(B,c)$ is not connected for some $B$ and $c$, then $G$ \notut{k}, for $k=|B|+2$. Unfortunately, the groups $\agl(1,p)$ have disconnected graphs for some $p$, and connected graphs for      
    other $p$. Therefore we need  sharper results.  The aim of this section is to prove them. 
    
  We start by providing a characterization of  connectedness in this setting.   
    We denote by $\GF(p)$ the field with $p$ elements and by $\GF(p)^{{*}}$ its non-zero elements.  
 \begin{prop}
 Let $G=\agl(1,p)$, with $p$ prime. If  $G(0,c)$ is not connected then $H:=\langle -1,c,c-1\rangle\subset \GF(p)^{{*}}$. In such a case, the orbit of $\{0,1,c\}$ under $\agl(1,p)$ has no section for the partition $\{\{0\},H,\mbox{rest}\}$. 
\end{prop}
\begin{proof}
Observe that $\{b,x,y\}$ is in the orbit of $\{0,1,c\}$ if and only if $$\{x,y\}\in \{\{\alpha+b,c\alpha+b\},\{b-\alpha, (c-1)\alpha +b\},\{(1-c)\alpha +b, b-c\alpha\}\mid \alpha \in \GF(p)^{{*}}\}.$$

In particular, for $b=0$, we have that 

$$\{x,y\}\in \{\{\alpha,c\alpha\},\{-\alpha, (c-1)\alpha\},\{(1-c)\alpha , -c\alpha\}\mid \alpha \in \GF(p)^{{*}}\}.$$

Now if $G(0,c)$ is not connected, then there exists 
a set $A\subset \GF(p)^{{*}}$ such that for every $\{x,y\}\in G(0,c)$ we have 
\[
x\in A \Leftrightarrow y \in A.
\]
 
In particular, for $\{x,y\}= \{\alpha,c\alpha\}$ we have $\alpha \in A \Leftrightarrow c\alpha \in A$, that is, $A=Ac$. In the same way we get the conditions $A=-A(c-1)$ and $Ac=A(c-1)$. 
Collecting these three conditions we get that: 
\[
A=Ac=A(c-1)=-A.
\] 

Now, for $c_{i}\in \{-1,c,c-1\}$, we have  $A=Ac_{1}=(Ac_{2})c_{1}=\ldots =Ac_{n}\cdots c_{1}$ and hence $A=A\langle -1,c,c-1\rangle$.
Clearly, saying that  {\em there exists a proper subset $A\subset \GF(p)^{{*}}$ such that $A=-A=Ac=A(c-1)$} is equivalent to  saying that {\em the group $\langle -1,c,c-1\rangle$ is strictly contained in $\GF(p)^{{*}}$}. The first claim follows.

Regarding the second claim, any section for the partition must have the form $\{0,c_{1},r\}$, with $c_{1}\in H$, and hence we must have $\{c_{1},r\}\in G(0,c)$. As we saw above, this means that 
$$\{c_{1},r\}\in \{\{\alpha,c\alpha\},\{-\alpha, (c-1)\alpha\},\{(1-c)\alpha , -c\alpha\}\mid \alpha \in \GF(p)^{{*}}\}.$$    
Checking all the possibilities always leads to the conclusion that $r\in H$. 
\end{proof}

The next result is our main result regarding the groups $\agl(1,p)$. 

\begin{thm}
Let $p$ be a prime and let $c\in \GF(p)\setminus \{0,1\}$. Then the following are equivalent:
\begin{enumerate}
\item the orbit of $\{0,1,c\}$ under $\agl(1,p)$ contains a section for all the $3$-partitions of $\{0,\ldots ,p\}$; 
\item $|\langle -1,c,c-1\rangle|=p-1$.
\end{enumerate}
\end{thm}
\begin{proof}
We already proved that if $|\langle -1,c,c-1\rangle|<p-1$, then there exists 
a 3-partition $P$ such that no set in the orbit of $\{0,1,c\}$ (under $\agl(1,p)$) is a section for $P$.

Conversely, suppose  that $|\langle -1,c,c-1\rangle|=p-1$ and suppose that there exists 
a bad partition $P=(A,C,A')$. This implies that $\Gamma(C,c)$ is not connected, that is, for all $\{x,y\}\in \Gamma(C,c)$, either $x,y\in A$ or $x,y\in A'$. As $\Gamma(C,c)$ is a union of graphs, this means  that for all $b\in C$, if $\{x,y\}\in G(b,c)$, then  $$x\in A \Leftrightarrow y \in A.$$

Now, repeating the arguments in the previous result we observe that this last condition is equivalent to saying that 
\[
\begin{array}{rccl}
\left(\forall b\in C,\alpha \in \GF(p)^{{*}}\right) &\alpha \in A-b &\Leftrightarrow & c\alpha\in A-b;\\
\left(\forall b\in C,\alpha \in \GF(p)^{{*}}\right) &-\alpha \in A-b &\Leftrightarrow & (c-1)\alpha \in A-b;\\
\left(\forall b\in C,\alpha \in \GF(p)^{{*}}\right) &(1-c)\alpha \in A-b &\Leftrightarrow & -c\alpha \in A-b.
\end{array}
\]

The first equivalence implies that $(A-b)=(A-b)c$, that is, $B=Bc$, for $B=A-b$; the second implies $B=B(c-1)$, and the last implies $B=-Bc^{{-1}}(1-c)^{{-1}}$. All these three together  imply  $B=Bc=B(c-1)=-B$. We already proved that this implies $|B|\geq |\langle -1,c,c-1\rangle|=p-1$, and clearly $|B|=|A|$. A contradiction since in $P=(A,C,A')$ the set $A$ cannot have $p-1$ elements.     
\end{proof}

If $p\equiv 1$ (mod $3$) and $p>7$, then we can take $c$ to be a primitive $6^{\mbox{th}}$ root of the unity; then $c^{2}=c-1$, so $\langle c,c-1,-1\rangle$ is a subgroup of order $6$. Thus $\agl(1,p)$ does not have the $3$-ut property if $p\equiv 1$ (mod $3$) and $p>7$. 

Also, if $p\equiv 1$ (mod $4$) and $p>5$, then there are consecutive quadratic residues in $\{1,\ldots, p-1\}$; if $c$ is the larger of such a pair, then $\langle c,c-1,-1\rangle$ is contained in the subgroup of squares and again $\agl(1,p)$ does not have the $3$-ut property. (If no two consecutive residues exist, then as $1$ and $p-1$ are residues, we see that residues and non-residues must alternate, apart from one pair of non-consecutive non-residues. But consecutive integer squares in $\{1,\ldots,p-1\}$ are an odd distance apart, and so there must be consecutive non-residues between them. So there are only two such squares, namely $1$ and $4$ and so $p=5$.)

Thus, only for primes $p\equiv 11$ (mod $12$) is the question undecided. 

We have not so far considered the subgroup of index $2$ in $\agl(1,p)$, which is $2$-homogeneous for $p\equiv 3$ (mod $4$). But if $p\equiv 7$ (mod $12$), then $\agl(1,p)$ does not have the $3$-ut property, and neither does its subgroup. So these are also undecided only for $p\equiv 11$ (mod $12$).

\subsection{The groups $\psl(2,q)$}

Regarding the groups $\psl(2,q) \le G \le\pgaml(2,q)$, with either $q$ prime (with the exception of
  $\psl(2,q)$ for $q\equiv1$ (mod~$4$), which are not $3$-homogeneous), or $q=2^p$ for $p$ prime, we have the following:
  \begin{enumerate}
\item Suppose $p$ is such that for some $c\in \GF(p)^{{*}}$ we have that $\langle -1,c,c-1\rangle$ is a proper subgroup of $\GF(p)^{{*}}$. Then   there exists 
a 3-partition $P=(A,B,C)$ such that the orbit of $\{0,1,c\}$ under $\agl(1,p)$ has no section for $P$. Therefore the partition $(\{\infty\},A,B,C)$ has no section in the orbit of $\{\infty
,0,1,c\}$ under $\pgl(2,p)$. This follows from the fact that any set in this orbit containing $\infty$ is of the form $\{\infty\}\cup D$ where $D$ is a 3-set in the orbit of $\{0,1,c\}$ under $\agl(1,p)$.    
\item By the previous observation, it follows that, for the same $p$,  $G$ \notut{4} for all $ G\leq \pgl(2,p)$.  
\item for the case $\psl(2,q)$, where $q=2^{{p}}$, we have $\psl(2,q)=\mbox{SL}(2,q)$.   
\item Below we have all the edges of $G(\{0,1\},c)$.
\end{enumerate}
\[
\begin{array}{l}
\left\{ {\frac {c{\alpha}^{2}}{c{\alpha}^{2}-c+1}},2\,{\frac {{\alpha}^{2}}{2\,{\alpha}^{2}-1}} \right\} 
 \left\{ 2\,{\frac {{\alpha}^{2}}{2\,{\alpha}^{2}+1}},2\,{\frac {c{\alpha}^{2}}{2\,c{\alpha}^{2}-c+2}} \right\} 
 \left\{ {\frac {c{\alpha}^{2}}{c{\alpha}^{2}-1+c}},2\,{\frac {c{\alpha}^{2}}{2\,c{\alpha}^{2}-2+c}} \right\} \\
 \left\{ {\frac {{\alpha}^{2} \left( {c}^{2}+2-3\,c \right) }{-3\,c{\alpha}^{2}+2\,{\alpha}^{2}+{c}^{2}{\alpha}^{2}-1}},{\frac {c{\alpha}^{2} \left( c-1 \right) }{1-c{\alpha}^{2}+{c}^{2}{\alpha}^{2}}}\right\} 
 \left\{ {\frac {{\alpha}^{2} \left( {c}^{2}+2-3\,c \right) }{-3\,c{\alpha}^{2}+2\,{\alpha}^{2}+{c}^{2}{\alpha}^{2}+1}},{\frac {c{\alpha}^{2} \left( c-2 \right) }{2-2\,c{\alpha}^{2}+{c}^{2}{\alpha}^{2}}}\right\} \\
\left\{ {\frac {c{\alpha}^{2} \left( c-1 \right) }{-c{\alpha}^{2}-1+{c}^{2}{\alpha}^{2}}},{\frac {c{\alpha}^{2} \left( c-2 \right) }{{c}^{2}{\alpha}^{2}-2\,c{\alpha}^{2}-2}} \right\} 
\left\{ {\frac {{\beta}^{2} \left( c-1 \right) }{c{\beta}^{2}-{\beta}^{2}-2+c}},{\frac {{\beta}^{2} \left( c-1 \right) }{c{\beta}^{2}-{\beta}^{2}-c}} \right\} \\
\left\{ {\frac {{\beta}^{2} \left( c-1 \right) }{c{\beta}^{2}-c-{\beta}^{2}+2}},{\frac {{\beta}^{2}}{{\beta}^{2}-2}} \right\} 
 \left\{ {\frac {{\beta}^{2} \left( c-1 \right) }{c{\beta}^{2}+c-{\beta}^{2}}},{\frac {{\beta}^{2}}{{\beta}^{2}+2}} \right\} 
\left\{ {\frac {{\beta}^{2} \left( c-2 \right) }{c{\beta}^{2}+4\,c-2\,{\beta}^{2}-4}},{\frac {{\beta}^{2}}{{\beta}^{2}+2}} \right\} \\
 \left\{ {\frac {{\beta}^{2} \left( c-2 \right) }{c{\beta}^{2}-2\,{\beta}^{2}+4-4\,c}},{\frac {{\beta}^{2} \left( c-2 \right) }{c{\beta}^{2}-2\,{\beta}^{2}-2\,c}} \right\} 
\left\{ {\frac {{\beta}^{2}}{{\beta}^{2}-2}},{\frac {{\beta}^{2} \left( c-2 \right) }{c{\beta}^{2}+2\,c-2\,{\beta}^{2}}} \right\} 
\end{array}
\]
  
  But it does not seem clear how an approach similar to the one used in $\agl(1,p)$ can be carried out here...
  
  \subsection{The state of the art} Regarding our list of exceptional cases the situation is the following:

\begin{enumerate}
\item $k=3$:
  \begin{enumerate}
  \item regarding $\agl(1,p)$, with $p$ prime, or its subgroup of index $2$ if $p \equiv 3$
 (mod~$4$), we have: 
 \begin{enumerate}
 \item $\agl(1,p)$, or (if $p\equiv 3$ (mod $4$)) its subgroup of index $2$, does not have the $3$-ut property, unless possibly when $p\equiv 11$ (mod $12$);  
 \item in general $\agl(1,p)$ \notut{3} if and only if there exists 
$c\in \GF(p)\setminus \{0,1\}$ such that $|\langle c,c-1,-1\rangle|<p-1$. 
\end{enumerate}
  \item The group $2^6 \colon \G_2(2)$   \ut{3}; the same happens to its subgroup of index $2$.
  \item The groups $\Sz(q)$ appear to have connected $G(b,c)$ and hence, probably, each one of them  \ut{3}.
  \item The Higman--Sims group \ut{3}.
  \end{enumerate}
\item $k=4$:
  \begin{enumerate}
  \item For the groups $\psl(2,q) \le G \le\pgaml(2,q)$, with either $q$ prime, or $q=2^p$ for $p$ prime, the situation is this:
  \begin{enumerate}
  \item if $q$ is prime and there exists 
$c\in GF(p)^{{*}}$ such that $|\langle c,c-1,-1\rangle|<q$, then  $G\leq \pgl(2,q)$ \notut{4}.
\item for $q\equiv1$ (mod~$4$) the group 
  $\psl(2,q)$  is not $3$-homogeneous. 
  \item what happens in the other groups is undecided.  
\end{enumerate}
 \item $\M_{11}$, degree $12$, \ut{4}.
  \end{enumerate}
\item $k=5$:
  \begin{enumerate}
  \item $\pgaml(2,32)$, degree $33$, \ut{5}.
  \end{enumerate}
\end{enumerate}

\section{Regular semigroups}

Arguably, three of the most important classes of semigroups are groups, inverse semigroups and regular semigroups, defined as follows: for a semigroup $S$ we have that
\begin{itemize}
\item  $S$ is a group if for all $a\in S$ there exists a unique $b\in S$ such that $a=aba$;
\item  $S$ is  inverse if for all $a\in S$ there exists a unique $b\in S$ such that $a=aba$ and $b=bab$;
\item  $S$ is regular if for all $a\in S$ there exists $b\in S$ such that $a=aba$.  
\end{itemize}

Recall from the introduction that to a large extent semigroup structure  theory is (almost) all about trying to show how the idempotents shape the structure of the semigroup. Therefore it is no surprise that groups and inverse semigroups can be  characterized by their idempotents:
\begin{itemize}
\item a semigroup is inverse if and only it is regular and the idempotents commute (see \cite[Theorem 5.5.1]{Ho95});  
\item a semigroup is a group if and only if it is regular and contains exactly one idempotent (see \cite[Ex. 3.11]{Ho95}). 
 \end{itemize}   
 
 Inverse semigroups, apart from being the class of (non-group) semigroups with the largest number of books dedicated to them, were introduced by geometers and they keep being very important to them \cite{Pa99}. 
 
   The full transformation semigroup $T(X)$ is regular, and all regular semigroups embed in some $T(X)$; every group embeds in some $T(X)$ as a group of permutations; and every inverse semigroup embeds in some $T(X)$ as an inverse semigroup of quasi-permutations, that is, transformations in which all but one of the kernel classes are singletons. (This follows from the Vagner--Preston representation \cite[Theorem 5.1.7]{Ho95} that maps every inverse semigroup into an isomorphic semigroup of partial bijections on a set; and every partial bijection $f$ on $X$ can be extended to a quasi-permutation $\bar{f}$ on $X\cup \{\infty\}$, defining $x\bar{f}=\infty$, for all $x$ not in the domain of $f$, and $x\bar{f}=xf$ elsewhere, yielding a semigroup of full quasi-permutations isomorphic to the original one.)    

 In the introduction we provided the classification of the groups $G\leq \sym$ such that the semigroup generated by $G$ and any map $a\in \trans$ is regular. Our aim now is to dramatically improve that result by extending it to quasi-permutations and transformations of a given rank. The main observation is the following straightforward lemma. 

\begin{lem}\label{auxl}
Let $a\in \trans$ and let $G\leq \sym$. Then $a$ is regular in $\langle a,G\rangle$ if and only if there exists $g\in G$ such that $\mbox{rank}(aga)=\mbox{rank}(a)$. 
\end{lem}
\begin{proof}
Suppose that $a$ is regular in $\langle a,G\rangle$. Then there exists 
$b\in \langle a,G\rangle$ such that $a=aba$. As $\mbox{rank}(uv)\leq \mbox{min}\{\mbox{rank}(u),\mbox{rank}(v)\}$ it follows that $\mbox{rank}(b)\geq \mbox{rank}(a)$. Now, either $b\in G$ and the result follows, or $b=g_{1}ag_{2}\ldots g_{m}ag_{m+1}$ and $\mbox{rank}(b)\leq \mbox{rank}(a)$, that is, $\mbox{rank}(a)=\mbox{rank}(b)$. Therefore, for every $g_{i}\in \{g_{2},\ldots,g_{m}\}$ we have $\mbox{rank}(ag_{i}a)=\mbox{rank}(a)$. It is proved that there exists 
$g\in G$ such that $\mbox{rank}(aga)=\mbox{rank}(a)$. 

Conversely, if $\mbox{rank}(aga)=\mbox{rank}(a)$, then $[n]ag$ is a transversal of $\mbox{Ker}(a)$ and hence $ga$ permutes $[n]a$. Therefore, for some natural $m$, $(ga)^{{m}}$ acts on $[n]a$ as the identity and hence $a(ga)^{{m}}=a$. The lemma follows. 
\end{proof}

In \cite{lmm} a stronger version of the previous result is proved. 

\begin{thm}(\cite[Theorem 2.3 and Corollary 2.4]{lmm})\label{mcalister}
Let $G\leq \sym$ and let $a\in \trans$. Then the following are equivalent:
\begin{enumerate}
\item there exists $g\in G$ such that $\mbox{rank}(aga)=\mbox{rank}(a)$;
\item $a$ is regular in $\langle G,a\rangle$;
\item every $b\in \langle G,a\rangle$, such that $\mbox{rank}(b)=\mbox{rank}(a)$, is regular in $\langle G,a\rangle$.
\end{enumerate}
\end{thm}

With the new tools developed in the previous sections we can now prove our first main theorem regarding regularity of semigroups generated by a group and a quasi-permutation. Recall that  $a\in \trans$ is a quasi-permutation if all, but one, of the $\mbox{Ker}(a)$-classes have one element.  
\begin{thm}\label{quasi-permutation}
Let $G\leq \sym$ and let $1< k< n$. Then the following are equivalent
\begin{enumerate}
\item every quasi-permutation $a$, such that $\mbox{rank}(a)=k$, is regular in  $\langle G,a\rangle$; 
\item $G$ is $(k-1)$-homogeneous or $G$ is one of the following groups
   \begin{enumerate}
    \item[\rm (a)] $n=5$ and $G\cong C_5,\ \dihed{5},$ with $k=2$;
    \item[\rm (b)]  $n=7$ and $G\cong\agl(1,7)$, with $k=3$;
   \item[\rm (c)] $n=9$ and $G\cong\asl(2,3)$ or $\agl(2,3)$, with $k=4$.
\end{enumerate}
\end{enumerate}
\end{thm}
\begin{proof}
Every quasi-permutation $a$ of rank $k$ has a kernel of the form $$(\{a_{1}\},\ldots,\{a_{k-1}\},\{a_{k},\ldots,a_{n}\})$$ and has image $\{b_{1},\ldots,b_{k}\}$. By the previous lemma we know that $a$ is going to be regular in $\langle a,G\rangle$ if and only if there exists 
$g\in G$ such that $\mbox{rank}(a)=\mbox{rank}(aga)$. But this is equivalent to saying that there exists 
$g\in G$ such that  $([n]\setminus  \{b_{1},\ldots,b_{k}\})g\subseteq \{a_{k},\ldots,a_{n}\}$. As these sets are arbitrary, it follows that $G$ satisfies the property that each quasi-permutation $a$ is regular in $\langle a,G\rangle$ if and only $G$ is $(n-k,n-k+1)$-homogenous. By Theorem \ref{2.1} this last condition is equivalent to $(2)$. It is proved that $(1)$ and $(2)$ are equivalent. 
\end{proof}

Now we can state and prove our second main result about quasi-permutations. 
\begin{thm}\label{quasi-permutation2}
Let $G\leq \sym$ and let $1< k\leq \lfloor \frac{n}{2}\rfloor$. Then the following are equivalent
\begin{enumerate}
\item for every quasi-permutation $a$, such that $\mbox{rank}(a)=k$, the semigroup $\langle G,a\rangle$ is regular;
\item $G$ is $(k-1)$-homogeneous or $G$ is one of the following groups
   \begin{enumerate}
    \item[\rm (a)] $n=5$ and $G\cong C_5,\ \dihed{5}$, with $k=2$;
    \item[\rm (b)]  $n=7$ and $G\cong\agl(1,7)$, with $k=3$;
\end{enumerate}
\end{enumerate}
\end{thm}
\begin{proof}
Clearly $(1)$ implies that every quasi-permutation $a$ is regular in $\langle a,G\rangle$ and hence, by the previous result, it follows that $G$ must be one of the groups listed in the statement of the theorem, or $n=9$ and $G\cong\asl(2,3)$ or $\agl(2,3)$. However, for $a:=\left(\begin{array}{cccc}\{1\}&\{2\}&\{3\}&\{4,\ldots,9\}\\ 1&4&5&2\end{array}\right)$, the semigroup $\langle a,G\rangle$ is not regular (when $G$ is $\asl(2,3)$ or $\agl(2,3)$). 
It is proved that $(1)$ implies $(2)$.

Conversely,  let $a$ be a rank $k$ quasi-permutation and let $G\leq \sym$ be a $(k-1)$-homogenous group. By the previous theorem ($(2)$ implies $(1)$) we know that $a$ is regular in $\langle a,G\rangle$ and hence, by Theorem \ref{mcalister}, every $b\in \langle a,G\rangle$ such that $\mbox{rank}(b)=k$ is regular in  $ \langle a,G\rangle$. Now suppose that $b\in  \langle a,G\rangle$ and $\mbox{rank}(b)=l<k$. Then $G$ is $l$-homogenous (because $k\leq \lfloor \frac{n}{2}\rfloor$) and hence there exists 
$g\in G$ such that $bgb$ has rank $l$. Thus, once again by Lemma \ref{auxl}, $b$ is regular in $ \langle b,G\rangle$. As $ \langle b,G\rangle\subseteq  \langle a,G\rangle$, it follows that $b$ is regular in $ \langle a,G\rangle$. 

That the groups listed in  (2) (a) and (b) satisfy the condition (1) follows from Theorem \ref{th2}. 
\end{proof}

Now we turn to the case of transformations of a given rank. 
Our main theorem is the following.  

\begin{thm}\label{Kut}
Let $n\geq 5$, $G\leq \sym$ and let $1< k\leq  \lfloor \frac{n+1}{2}\rfloor$. Then
the following are equivalent:
\begin{enumerate}
	\item for all rank $k$ transformations $a\in \trans$, we have that $a$ is regular in $\langle a, G\rangle$;
	\item for all rank $k$ transformations $a\in \trans$, the semigroup $\langle a, G\rangle$ is regular;
	\item $G$ has the $k$-ut property (and hence is one of the groups listed in Theorem \ref{RegularAndKut}).
\end{enumerate}
\end{thm}
\begin{proof}
It follows from Lemma \ref{auxl} that (1) and (3) are equivalent, and  (2) implies (1) trivially. In addition, (1) (together with Theorem \ref{mcalister}) implies that $b$ is regular in $\langle a,G\rangle$,  for all $b\in \langle a,G\rangle$, with $\mbox{rank}(b)=\mbox{rank}(a)$. It also follows from Proposition \ref{critical} (2) that 
 if $G$ has the $k$-ut property, then $G$ is $(k-1)$-homogenous, or it is one of the exceptions in Theorem \ref{th6}. If $G$ is $(k-1)$-homogenous, then  
  for every $c\in \langle a,G\rangle$,  with $\mbox{rank}(c)<\mbox{rank}(a)$, we have that $G$ is $\mbox{rank}(c)$-homogenous and hence $c$ is regular in $\langle c,G\rangle\subseteq \langle a,G\rangle$. Thus $c$ is regular in $\langle a,G\rangle$. 
  
  If $G$ is one of the three exceptions of degree $5$ and $7$, then (by Theorem \ref{th6c})  $\langle a,G\rangle$ is regular, for all rank$(k)$ maps $a$. Thus it is proved that (1) and (3) imply (2). \end{proof}

When $n=9$ and $G= \asl(2,3)$ or $G=\agl(2,3)$, (by Theorem \ref{quasi-permutation2}) there exist 
rank-$4$ maps $a\in \trans$ such that $\langle a,G\rangle$ is not regular. Also these groups do not have the $4$-ut property.

\section{Problems}

We start by proposing a problem to experts in number theory. If this problem can be solved, the results on $\agl(1,p)$, in Section \ref{numbertheory}, will be dramatically sharpened. 
\begin{problem}
Classify the prime numbers $p$ congruent to $11$ (mod $12$) such that for some $c\in \GF(p)^{{*}}$ we have $|\langle -1,c,c-1\rangle|<p-1$. 
\end{problem}

\begin{problem}
Do the Suzuki groups $\Sz(q)$ have the $3$-ut property? 

Classify the groups $G$ that have the $4$-ut property, when  $\psl(2,q) \le G \le\pgaml(2,q)$, with either $q$ prime (except 
  $\psl(2,q)$ for $q\equiv1$ (mod~$4$), which is not $3$-homogeneous), or $q=2^p$ for $p$ prime.
\end{problem}

\begin{problem}
Prove results analogous to Theorem \ref{quasi-permutation2} and to Theorem 	\ref{Kut}, but for the transformations of rank $k>\lfloor \frac{n+1}{2}\rfloor$. 
\end{problem}	
The difficulty here (when rank $k>\lfloor \frac{n+1}{2}\rfloor$) is that a $k$-homogenous group is not necessarily $(k-1)$-homogenous. Therefore a rank $k$ map $a\in \trans$ might be regular in $\langle a, G\rangle$, but we are not sure that there exists 
$g\in G$ such that $\mbox{rank}(bgb)=\mbox{rank}(b)$, for $b\in  \langle a, G\rangle$ such that $\mbox{rank}(b)<\mbox{rank}(a)$.

\begin{problem}\label{top}
In what concerns this paper, the most general problem that has to be handled is the classification of pairs $(a,G)$, where $a\in \trans$ and $G\leq \sym$, such that $\genset{a,G}$ is a regular semigroup.
\end{problem}

When investigating $(k-1)$-homogenous groups without the $k$-ut property, it was common that some of the orbits on the $k$-sets have transversals for all the partitions. Therefore the following definition is natural. 
A group $G\leq \sym$ is said to have the weak $k$-ut property if there exists 
a $k$-set $S\subseteq [n]$ such that the orbit of $S$ under $G$ contains a section for all $k$-partitions. And such a set is called a $G$-universal transversal set. 
\begin{problem}\label{five}
Classify the groups with the weak $k$-ut property; in addition,  for each one of them, classify their $G$-universal transversal sets. 
\end{problem}

In this paper we considered groups such that the orbit of every $k$-set contains a section for every $k$-partition. And this is of course a very strong requirement. In order to attack Problem \ref{top}, it seems the next step (in addition to Problem \ref{five}) is to consider groups such that the orbit of every $k$-set contains sections for some (not all) partitions.

\begin{problem}
Let  $\pi$ be a partition of $n$. A map $a\in \trans$ has kernel type $\pi$ if the partition of $n$ induced by the cardinalities of the kernel blocks is equal to $\pi$.  Classify the groups $G\leq \sym$ such that for all maps $a\in \trans$ of a given kernel type $\pi$, the semigroup $\genset{G,a}$ is regular.
\end{problem}

In McAlister's celebrated paper \cite{mcalister} it is proved that if $e^2=e\in \trans$ is a rank $n-1$ idempotent, then $\genset{G,e}$ is regular for all groups $G\leq \sym$. In addition, assuming  that $\{\alpha,\beta\}$ is the non-singleton kernel class of $e$ and $\alpha e=\beta$,  if $\alpha$ and $\beta$ are not in the same orbit under $G$, then   $\genset{e,G}$ is an orthodox semigroup (that is, the idempotents form a subsemigroup); and $\genset{G,e}$ is inverse if and only if $\alpha$ and $\beta$ are not in the same orbit under $G$ and the stabilizer of $\alpha$ is contained in the stabilizer of $\beta$.

\begin{problem}
Classify the groups $G\leq \sym$ that together with any idempotent [rank $k$ idempotent] generate a regular [orthodox, inverse] semigroup.

Classify the pairs $(G,a)$, with $a\in \trans$ and $G\leq \sym$, such that $\genset{G,a}$ is inverse [orthodox]. (Recall that by \cite{schein} every element $a\in \trans$ is contained in an inverse subsemigroup of $\trans$; in addition it is a longstanding open problem to describe the maximal inverse subsemigroups of $\trans$.)
\end{problem}

A group $G\leq \sym$ has the $(n-1)$-universal transversal property if and only if it is transitive. And $\genset{G,a}$ contains all the rank $n-1$ maps of $\trans$ if and only if $G$ is $2$-homogeneous. In this last case $\genset{G,a}$ is regular for all $a\in \trans$ because $\genset{G,a}=\{b\in \trans\mid |[n]b|\leq n-1\}\cup G$, and this semigroup is well known to be regular.

\begin{problem}
Classify the groups $G\leq \sym$ such that $G$ together with any  rank $n-k$ map, where $k\leq 5$,  generate a regular semigroup. We already know that such $G$ must be $k$-homogeneous and so are classified. 
\end{problem}

The majority of the previous problems (and theorems) admit an obvious analogous with regular replaced everywhere by {\em idempotent generated}.

\begin{problem}
Classify all the pairs $(a,G)$, where $a\in \trans$ and $G\leq \sym$, such that $\genset{a,G}\setminus G$ is idempotent generated (that is, $\genset{a,G}\setminus G$ is generated by its own idempotents).

Solve particular instances of this general problem analogous to the list of problems above.
\end{problem}

The theorems and problems in this paper admit linear versions that are interesting for experts in groups and semigroups, but also to experts in linear algebra and matrix theory. However, for the linear case, not even an analogue of Theorem \ref{th2} exists. All we know is that any singular matrix with any group containing the  special linear group generate a regular semigroup \cite{ArSi1,ArSi2} (see also the related papers \cite{Gr,Pa,Ra}).

\begin{problem}
Prove (or disprove) that if $G\leq GL(n,q)$  such that for all singular matrix $a$ there exists 
$g\in G$ with  $\rank(a)=\rank(aga)$, then $G$ contains the special linear group.
\end{problem}

It is clear that such a group must satisfy the following property. If $V$ is a vector space (over a finite field) with $\dim(V)=n$, and $U,T\leq V$ are two non-null subspaces such that $\dim(U)+\dim(T)=n$, then there exists 
$g\in G$ such that $V=Ug\oplus T$. 

For $n=2$ and for $n=3$, this condition is equivalent to irreducibility of $G$. But we conjecture that, for sufficiently large $n$, it implies that $G$ contains the special linear group. 

\begin{problem}\label{11}
Classify the groups $G\leq GL(n,q)$  such that for all rank $k$ (for a given $k$) singular matrix $a$ we have that $a$ is regular in $\genset{G,a}$ [the semigroup $\genset{G,a}$ is regular].
\end{problem}

To handle this problem it is useful to keep in mind the following results. Kantor~\cite{kantor:inc} proved that if a subgroup of $\pgaml(d,q)$ acts transitively on $k$-dimensional subspaces, then it acts transitively on $l$-dimensional subspaces for all $l\le k$ such that $k+l\le n$; in~\cite{kantor:line}, he showed that subgroups transitive on $2$-dimensional subspaces are $2$-transitive on the $1$-dimensional subspaces with the single exception of a subgroup of $\pgl(5,2)$ of order $31\cdot5$; and, with the second author~\cite{cameron-kantor}, he showed that such groups must contain $\psl(d,q)$ with the single exception of the alternating group $A_7$ inside $\pgl(4,2)\cong A_8$. Also Hering \cite{He74,He85} and Liebeck \cite{Li86} classified the subgroups of $\pgl(d,p)$ which are transitive on $1$-spaces. 

\begin{problem}
Solve the analogue of Problem \ref{11} for independence algebras (for definitions and fundamental results see \cite{ArEdGi,arfo,cameronSz,gould}). 
\end{problem}

Recall from Subsection \ref{general} the graph $G(B,c)$, where $G$ is a $t$-homogeneous group and $|B|=t-1$.

\begin{problem}
Is it true  that the group $G$  has the $(t+1)$-ut property if and only if $G(B,c)$ is connected? If so, is it possible to find an elementary proof of that (without using the classification of finite simple groups)? 
 \end{problem}
 
 \begin{problem}
 Prove Corollary \ref{cor1} and Corollary \ref{cor2} without using the classification of finite simple groups. 
 \end{problem}
 
 Regarding this problem, observe that  the proof  that $k$-ut implies $(k-1)$-ut in fact follows from the classification of the $(k-1,k)$-homogenous groups (and hence, for that purpose, we can bypass the classification of groups with the $k$-ut property).  In fact,  if the group possesses the $k$-ut property (for $k\leq \lfloor \frac{n}{2}\rfloor$), then it is $(k-1,k)$-homogeneous and hence, with few exceptions, it is 
$(k-1)$-homogeneous so that it has the $(k-1)$-ut property. So in the previous theorem what really is at stake is to find an elementary proof to Corollary \ref{cor2}. 
\def\cprime{$'$}

\section*{Acknowledgements}
We gratefully thank  various conversations with P. M. Neumann on the early stages of this investigation.  We also thank the developers of GAP \cite{GAP} and 
Soicher for GRAPE \cite{So06}.

The first author was partially supported by FCT and FEDER, Project POCTI-ISFL-1-143 of Centro de Algebra da Universidade de Lisboa, by FCT and PIDDAC through the project PTDC/MAT/69514/2006, by PTDC/MAT/69514/2006 Semigroups and Languages, and by  PTDC/MAT/101993/2008 Computations in groups and semigroups.

\end{document}